\documentclass[12pt]{amsart}

\usepackage{amsmath,amssymb,color,slashed}
\usepackage{amsthm}
\usepackage{tikz,wrapfig}
\usepackage[colorlinks = true, linkcolor = blue, citecolor = blue, anchorcolor = blue]{hyperref}
\usepackage[margin=1.0in]{geometry}
\usepackage{tensor}
\usepackage{dutchcal,tikz-cd}
\usepackage[mathscr]{eucal}

\numberwithin{equation}{section}

\newtheorem{prop}{Proposition}
\newtheorem{lemma}[prop]{Lemma}

\newtheorem{thm}[prop]{Theorem}
\newtheorem{cor}[prop]{Corollary}

\numberwithin{prop}{section}

\theoremstyle{definition}
\newtheorem{defn}[prop]{Definition}
\newtheorem{ex}[prop]{Example}
\newtheorem{rmk}[prop]{Remark}
\newtheorem{ques}[prop]{Question}
\newtheorem{nota}[prop]{Notation}

\newcommand{\del}{\partial}
\newcommand{\dt}{\frac{\partial}{\partial t}}
\newcommand{\tdt}{\tfrac{\partial}{\partial t}}
\newcommand{\brs}[1]{\left| #1 \right|}

\newcommand{\gG}{\Gamma}
\renewcommand{\gg}{\gamma}
\newcommand{\gD}{\Delta}

\newcommand{\gs}{\sigma}

\newcommand{\gk}{\kappa}
\newcommand{\gl}{\lambda}

\newcommand{\ga}{\alpha}
\newcommand{\gb}{\beta}

\renewcommand{\ge}{\epsilon}
\newcommand{\N}{\nabla}

\newcommand{\BB}{\mathcal B}

\newcommand{\LL}{\mathcal L}

\newcommand{\til}[1]{\widetilde{#1}}

\renewcommand{\bar}[1]{\overline{#1}}

\newcommand{\IP}[1]{\left<#1\right>}

\DeclareMathOperator{\gm}{\mathcal{G}}
\DeclareMathOperator{\gr}{\mathcal{R}}
\DeclareMathOperator{\Rc}{Rc}
\DeclareMathOperator{\rc}{Rc}

\DeclareMathOperator{\grc}{\mathcal{R\hspace{-.4mm}c}}
\DeclareMathOperator{\dil}{\mathcal{D\hspace{-.4mm}il}}
\DeclareMathOperator{\fgrc}{\bar{\mathcal{R\hspace{-.4mm}c}}}

\DeclareMathOperator{\grm}{\mathcal{R\hspace{-.3mm}m}}
\newcommand{\differential}{\mathcal D}
\newcommand{\eh}{\ms S}

\DeclareMathOperator{\tr}{tr}

\DeclareMathOperator{\Id}{Id}
\DeclareMathOperator{\divg}{div}

\DeclareMathOperator{\Isom}{Isom}
\DeclareMathOperator{\End}{End}
\DeclareMathOperator{\Aut}{Aut}

\newcommand{\on}{\operatorname}
\renewcommand{\div}{\on{div}}

\newcommand{\ms}{\mathscr}
\newcommand{\mc}{\mathcal}
\newcommand{\mf}{\mathfrak}
\newcommand{\mb}{\mathbb}

\newcommand{\w}[1]{\wedge^{\!#1}\,}

\newcommand{\slot}{\cdot}

\newcommand{\at}{\on{At}}

\newcommand{\tdels}{\tfrac{\del}{\del s}}
\newcommand{\dels}{\frac{\del}{\del s}}

\definecolor{imperialblue}{RGB}{0,62,116}

\usepackage{marginnote}

\allowdisplaybreaks

\begin{document}

\title[Ricci flow on Courant algebroids]{Ricci flow on Courant algebroids}

\begin{abstract} We develop a theory of Ricci flow for metrics on Courant algebroids which unifies and extends the analytic theory of various geometric flows, yielding a general tool for constructing solutions to supergravity equations.  We prove short time existence and uniqueness of solutions on compact manifolds, in turn showing that the Courant isometry group is preserved by the flow.  We show a scalar curvature monotonicity formula and prove that generalized Ricci flow is a gradient flow, extending fundamental works of Hamilton and Perelman.  Using these we show a convergence result for certain nonsingular solutions to generalized Ricci flow.
\end{abstract}

\date{\today}

\author{Jeffrey Streets}
\address{J.\,Streets \\ Rowland Hall\\
        University of California\\
        Irvine, CA 92617}
\email{\href{mailto:jstreets@uci.edu}{jstreets@uci.edu}}

\author{Charles Strickland-Constable}
\address{C.\,Strickland-Constable \\ Department of Physics, Astronomy and Mathematics, University of Hertfordshire, College Lane, Hatfield, AL10 9AB, UK}
\email{\href{mailto:c.strickland-constable@herts.ac.uk}{c.strickland-constable@herts.ac.uk}}

\author{Fridrich Valach}
\address{F.\,Valach \\ Department of Physics, Astronomy and Mathematics, University of Hertfordshire, College Lane, Hatfield, AL10 9AB, UK}
\email{\href{mailto:fridrich.valach@gmail.com}{fridrich.valach@gmail.com}}

\maketitle

\section{Introduction}

The Ricci flow is a central tool in understanding the geometry and topology of manifolds, with myriad applications.  In recent years, motivated by problems in complex geometry and mathematical physics, various geometric flows have emerged, in many cases motivated by the goal of constructing solutions to supergravity equations \cite{CollinsPhongSpinorflows, Feiparabolicsupergravity, GF19, OSW, PhongAnomaly,Ashmore:2023ift,Streetsexpent}.  A particularly motivating example for the present work is the \emph{generalized Ricci flow}.  A one-parameter family $(g_t, H_t)$ of Riemannian metrics and closed three-forms solves generalized Ricci flow if
\begin{gather} \label{f:exactGRF}
\begin{split}
    \dt g =&\ -2 \Rc + \tfrac{1}{2} H^2, \qquad \dt B = - d^*_g H, \qquad H = H_0 + dB,
    \end{split}
\end{gather}
where $H^2(X,Y) = \IP{i_X H, i_Y H}_g$.  These equations arise as the one-loop renormalization group flow for a certain nonlinear sigma model \cite{Friedanetal, Polchinski}, and they were shown to be a gradient flow using the string effective action in \cite{OSW}.  Fundamental existence and regularity properties for this equation were shown in \cite{Streetsexpent}, and natural conditions in complex geometry were shown to be preserved by this equation yielding pluriclosed flow \cite{PCF, PCFReg} and generalized K\"ahler-Ricci flow \cite{GKRF}.  Recently, by incorporating  a solution to an auxiliary scalar dilaton flow, a scalar curvature monotonicity was found \cite{Streetsscalar}.

Many of the geometric properties and analytic estimates for generalized Ricci flow arise from the relatively new subject of \emph{generalized geometry}.  The fundamental objects in this subject are Courant algebroids (cf.\ \S \ref{s:Courant}), which are vector bundles equipped with certain natural structures motivated by dynamical systems and Poisson geometry \cite{LWX,CourantDirac,Severa}, which have found applications in various subjects, including capturing Hitchin/Gualtieri's notion of a generalized complex structure \cite{HitchinGCY, GualtieriGCG}.  
Furthermore, the background equations for the bosonic string as well as the bosonic equations of motion of heterotic supergravity theories arise naturally as the Ricci-flatness condition for a generalized metric on a Courant algebroid~\cite{CSCW,GF14,Coimbra:2014qaa}. 
Key examples include supersymmetric solutions such as the well-known Hull--Strominger system \cite{Hullsystem,StromingerSST}, numerous further properties of which can be expressed in terms of Courant algebroids~\cite{delaOssa:2014cia,Ashmore:2018ybe,Garcia-Fernandez:2018ypt,Garcia-Fernandez:2020awc,Ashmore:2019rkx}. 
In \cite{StreetsTdual, GRFbook} it was shown that generalized Ricci flow \eqref{f:exactGRF} is equivalent to a flow of generalized metrics on an exact Courant algebroid, whereas the pluriclosed flow is equivalent to a natural flow of Hermitian metrics on a holomorphic Courant algebroid \cite{JordanStreets,JGFS}.  

Our goal in this paper is to initiate the analytic study of Ricci flow for generalized metrics on Courant algebroids, extending and unifying the prior results described above.  This equation, which we will still call \emph{generalized Ricci flow}, was introduced in \cite{SeveraValach1,GF19}.  As described in detail in \cite{GF19}, it turns out that a generalized metric alone does not determine a unique Levi-Civita connection, nor thus a unique Ricci tensor, as for the cases considered in~\cite{Siegel,CSCW}.  However, by imposing compatibility with a divergence operator, one obtains a class of Levi-Civita connections which has uniquely determined Ricci and scalar curvatures.  Frequently, we will impose that the divergence operator is associated to a smooth function, called the dilaton, which for technical reasons we will often identify with a section $\gs$ of the space of half-densities on the underlying manifold (cf.\ \S \ref{s:Courant}).  Thus one is lead to a natural generalization of the Einstein-Hilbert functional \cite{SeveraValach2}.  By imposing a particular choice of mixed-signature inner product on the space of generalized metrics and half-densities motivated by mathematical physics (cf.\ Definition \ref{d:L2metric}), this functional has formal gradient flow\footnote{This is distinct from the $\gl$-functional monotonicity, detailed below.} given by the generalized Ricci flow:
\begin{gather} \label{f:GRF}
\dt \gm = -2 \grc_{\gm,\gs}, \qquad \dt \gs = - \tfrac{1}{2} \gr_{\gm,\gs} \gs.
\end{gather}
In view of the discussion above, this flow is a natural tool for constructing solutions to supergravity equations, and canonical geometric structures on manifolds.  We note that in a certain sense the flow for $\sigma$ can be gauged away, resulting purely in a flow of generalized metrics, although it is decidedly less natural to do this (cf.\ Remark \ref{r:halfdensitygauge}).

Our first result establishes short-time existence and uniqueness of solutions on compact manifolds:

\begin{thm} \label{t:mainthm1} Given $E \to M$ a Courant algebroid over a compact manifold and $(\gm_0, \gs_0)$ a generalized metric on $E$ and half-density on $M$, there exists $\ge > 0$ and a unique solution to generalized Ricci flow with initial condition $(\gm_0,\gs_0)$ on $[0, \ge)$.
\end{thm}

The proof of short-time existence follows the DeTurck gauge-fixing method.  The generalized Ricci tensor is invariant under the infinite dimensional group of Courant automorphisms.  This invariance renders the Ricci tensor degenerate elliptic, thus rendering the generalized Ricci flow degenerate parabolic.  We show that this degeneracy can be overcome by adding a term tangent to the action of Courant automorphisms, and then pull back the solution to this equation by the relevant family of automorphisms to produce the required solution to generalized Ricci flow.  The uniqueness proof follows Hamilton's approach for Ricci flow which recasts the DeTurck diffeomorphisms as solutions to harmonic map heat flow with respect to the flowing metric.  Here a further subtlety arises as the relevant operator on Courant automorphisms has a degeneracy due to the infinite dimensional space of exact differentials, which have trivial infinitesimal action (cf.\ Proposition \ref{prop:sections_automorphisms}, Lemma \ref{l:Wlinearization}).  We apply a further gauge-fixing argument to overcome this degeneracy to obtain the uniqueness.

As a corollary we obtain that the Courant isometry group of the initial data is preserved along the flow.
\begin{cor} Given $(\gm_t,\gs_t)$ a solution to generalized Ricci flow on a compact manifold, $\Isom(\gm_0) \leqslant \Isom(\gm_t)$.
\end{cor}
\noindent Furthermore, the concept of T-duality is compatible with generalized Ricci flow (cf. Remark \ref{r:Tduality}).

To sharpen the analytic theory, we establish a series of monotonicity results for generalized Ricci flow.  These results require a slightly stronger hypothesis on the generalized metric, namely that it is positive definite, a condition we call \emph{strict positivity} (cf.\ Definition \ref{d:genmetric}, Remark \ref{r:strictpositive}).  Our second main result establishes a scalar curvature monotonicity formula for generalized Ricci flow.  In the case of exact Courant algebroids such a formula was recently established \cite{Streetsscalar}, incorporating a further scalar field, corresponding to the physical dilaton.
 
\begin{thm} \label{t:scalarevolutionintro} (cf.\ Theorem \ref{t:scalarcurvmon})
            Let $(\gm_t, \gs_t)$ be a solution to generalized Ricci flow on a compact manifold.  Then for all smooth existence times $t$ one has
            \[\dt\gr= \gD_{\gm,\gs} \gr+\,\vert\!\grc\!\vert_{\gm}^2.\]
\end{thm}
\noindent An immediate corollary is that a lower bound on scalar curvature is preserved for generalized Ricci flows with strictly positive initial data (cf.\ Corollary \ref{c:scalarlb}).  

Another fundamental feature of the Ricci flow equation is that it is a gradient flow \cite{Perelman1}.  Perelman's discovery of this gradient flow property exploits the $\mathcal F$-functional, which is a special case of the Einstein-Hilbert functional/string action described above.  This functional naturally incorporates a scalar weight $e^{-f}$, which corresponds here to an arbitrary choice of positive half-density.  The general definition extending Perelman's $\gl$ functional is then
\begin{align*}
\gl(\gm) := \inf_{ \{ \gs \in \ms H^* \mid\, \int \gs^2 = 1 \}} \eh(\gm, \gs),
\end{align*}
where $\ms H^*$ is the space of positive half-densities.  Our next main result is then:
\begin{thm} \label{t:gradientintro} (cf.\ Theorem \ref{t:gradient})
  For strictly positive metrics, generalized Ricci flow is the gradient flow of $\gl$.  More precisely, given $(\gm_t, \gs_t)$ a solution to generalized Ricci flow with strictly positive initial data over a compact manifold $M$, for any smooth existence times $t_1 < t_2$ one has $\gl(\gm_{t_1}) \leq \gl(\gm_{t_2})$.  Furthermore, equality holds if and only if $\gm_{t}$ is a soliton.
\end{thm}
\noindent Theorem \ref{t:gradientintro} was established for Ricci flow in \cite{Perelman1}, and was extended to the generalized Ricci flow (\ref{f:exactGRF}) in \cite{OSW}.  Insofar as generalized Ricci flow corresponds to physical renormalization group flow, this result is a concrete manifestation of Zamolodchikov's $c$-theorem \cite{zamolodchikov}.

In fact Theorem \ref{t:gradientintro} is a consequence of a more general differential inequality for the weighted scalar curvature arising from a solution to the relevant conjugate heat equation.  More precisely, associated to a solution of generalized Ricci flow we have a time-dependent scalar heat operator $\square = \dt - \gD_{\gm,\gs}$, where the Laplace operator is naturally associated to the data $\gm, \gs$, and turns out to be the weighted Laplacian associated to the underlying classic Riemannian metric $g$ and the dilaton (cf.\ Lemma \ref{l:weightedlaplacian}).  Furthermore we obtain a conjugate heat operator $\square^*_{\gs}$ relative to the measure $\gs^2$, which turns out to be the backwards heat operator with reaction term $\gr_{\gm,\gs}$.  Given a solution $u = e^{-f}$ of the conjugate heat equation $\square^*_{\gs} u = 0$, it then turns out that
\begin{align} \label{f:introharnack}
    \square^*_{\gs} \left( \gr(\gm, \gs^f) u \right) =&\ - \brs{\grc(\gm, \gs^f)}^2_{\gm} u \leq 0.
\end{align}
where as above $\gs^f = e^{-f/2} \gs$ (cf.\ Proposition \ref{p:steadyharnack}).  By a standard argument one obtains the monotonicity of $\gl$ from this equation.  

Perelman's derivation of the $\gl$-monotonicity introduced many new ideas into the study of Ricci flow such as the weighted Ricci and scalar curvatures, conjugate heat operator, etc., and brilliantly combined them to derive the key differential inequality above.  In \cite{Streetsscalar} the first author showed an extension of this differential inequality for generalized Ricci flow on exact Courant algebroids (\ref{f:exactGRF}), crucially using the coupling to the dilaton flow.  Surprisingly, these subtle constructions are quite natural in the language of generalized geometry on Courant algebroids, leading to a simple proof of (\ref{f:introharnack}).  In particular, to even define the Ricci tensor requires a divergence operator, represented here by the half-density $\gs$.  The evolution of $\gs$, derived by general principle from the action as described above, induces the key dilaton flow used in \cite{Streetsscalar}, and derived from physical principles \cite{Polchinski}.  Modifying this divergence operator, i.e.\ weighting the half-density by $e^{-f/2}$, naturally corresponds to modifying the Ricci tensor by a generalized Lie derivative (Proposition \ref{p:Riccidivergencechange}).  If $u = e^{-f}$ is a solution of the conjugate heat equation, and we gauge-fix the flow to remove this Lie derivative term, it follows that the half-density $\gs^f$ is fixed.  Using this, the variation formula for the generalized scalar curvature with arbitrary divergence operator, as well as the general differential Bianchi identity for Ricci and scalar, again with arbitrary divergence operator (Proposition \ref{prop:second_contracted_bianchi}), then yields a very simple proof of (\ref{f:introharnack}).

As a final corollary of the $\gl$-monotonicity we obtain that global solutions to generalized Ricci flow which are nonsingular in a certain sense (cf.\ Definition \ref{d:nonsingular}) converge to solitons.
\begin{cor} \label{c:convcor} (cf.\ Corollary \ref{c:convcor2}) Every nonsingular solution to generalized Ricci flow with strictly positive initial data converges subsequentially to a soliton.
\end{cor}

\textbf{Acknowledgements:} JS was supported by the NSF via DMS-2203536 and by a Simons Fellowship. CS-C and FV are supported by EPSRC grant EP/X014959/1. FV was also supported by the Postdoc Mobility grant P500PT\_203123 of the Swiss National Science Foundation.  The authors thank Andr\'e Coimbra for earlier collaboration on the linearisation of curvature tensors in generalized geometry. We also thank Mario Garcia-Fernandez for comments on an earlier draft of this paper, and Daniel Waldram and Pavol \v Severa for helpful discussions. The authors also thank the organisers of the ``Supergravity, Generalized Geometry and Ricci Flow'' programme at the Simons Center for Geometry and Physics at which some of the research for this paper was performed.

\section{Courant algebroids} \label{s:Courant}

In this section we recall fundamental aspects of the theory of Courant algebroids.  We starting  by recalling the basic definition \cite{LWX} and the notion of a Courant algebroid automorphism.  We then recall basic aspects of the theory of Courant algebroid automorphisms and generalized Lie derivatives, following \cite{AXu, GualtieriBranes, GF14,SeveraPLTDuality}.

\subsection{Definition and first examples}

\begin{defn}\label{def:courant}
    A \emph{Courant algebroid} is a vector bundle $E\to M$ equipped with:
    \begin{itemize}
        \item an $\mb R$-bilinear bracket $[\cdot,\cdot]\colon \Gamma(E)\times \Gamma(E)\to\Gamma(E)$,
        \item a fiberwise non-degenerate symmetric bilinear form $\langle \cdot,\cdot\rangle$,
        \item a vector bundle map (the \emph{anchor}) $\rho\colon E\to TM$,
    \end{itemize}
    satisfying the following properties for all $u,v,w\in\Gamma(E)$, $f\in C^\infty(M)$:
    \begin{enumerate}
        \item $[u,[v,w]]=[[u,v],w]+[v,[u,w]]$,
        \item $[u,fv]=f[u,v]+ (\rho(u)f)v$,
        \item $\rho(u)\langle v,w\rangle=\langle [u,v],w\rangle+\langle v,[u,w]\rangle$,
        \item $[u,v]+[v,u]=\differential\langle u,v\rangle$.
    \end{enumerate}
    where $\differential:=\rho^T\circ d\colon C^\infty(M)\to \Gamma(E^*)$ and we have identified $E\cong E^*$ via the pairing.
\end{defn}
Combining the first two axioms, it is easy to show that the anchor intertwines the brackets on $E$ and $TM$, i.e.
\[\rho([u,v])=[\rho(u),\rho(v)].\]
Applying this property to the last axiom we obtain that $\rho\circ \rho^T=0$, or in other words, we have a chain complex
\begin{align}\label{eq:chain}
    0\to T^*M\xrightarrow{\rho^T}E\xrightarrow{\rho}TM\to 0.
\end{align}
Also, setting $v=u$ in the first axiom, we get \[[\differential f,w]=0.\]
\begin{defn} \label{d:tCA}
    If \eqref{eq:chain} is an exact sequence, the Courant algebroid is called \emph{exact}. More generally, if $\rho$ is surjective, the algebroid is called \emph{transitive}.
\end{defn}
\begin{ex}\label{ex:exact}
    Let $H\in \Omega^3(M)$ be a closed 3-form. We then have an exact Courant algebroid structure on $E=TM\oplus T^*M$ given by
    \[\langle X+\alpha,Y+\beta\rangle=\alpha(Y)+\beta(X),\quad [X+\alpha,Y+\beta]=L_X Y+(L_X \beta-i_Y d\alpha+H(X,Y,\cdot)).\]
    Conversely \cite{Severa}, for every exact Courant algebroid there exists a (non-canonical) splitting of \eqref{eq:chain} which results in a pairing and bracket of this form for some $H$. A change of splitting corresponds to the shift of $H$ by an exact 3-form, yielding the classification of exact Courant algebroids in terms of $H^3(M)$.
\end{ex}
\begin{ex}\label{ex:trivial}
    A Courant algebroid over $M=\text{point}$ is the same as a \emph{quadratic Lie algebra}, i.e.\ a Lie algebra with an invariant non-degenerate symmetric bilinear form. (We have $\rho=0$.)
\end{ex}

\subsection{Courant algebroid automorphisms} \label{ss:CAA}

Let $E\to M$ be a Courant algebroid.  Suppose that $\Phi$ is a vector bundle automorphism of $E$, covering a diffeomorphism of the base $M$, which we shall denote by $\bar \Phi$. Then $\Phi$ has a well-defined action on the sections of $E$, given by
\[\Phi^* u:=\Phi^{-1}\circ u \circ \bar\Phi.\]

\begin{defn}[\cite{BursztynCavalcantiGualtieri}]
    A \emph{Courant algebroid automorphism} is a vector bundle automorphism $\Phi$ of $E$, covering a diffeomorphism $\bar\Phi$ of $M$ such that for all $u,v\in\Gamma(E)$ we have
    \[\langle \Phi^* u,\Phi^*v\rangle=\bar\Phi^{\,*}\langle u,v\rangle,\qquad \rho(\Phi^*u)=\bar\Phi_*\rho(u),\qquad \Phi^*[u,v]=[\Phi^*u,\Phi^*v].\]
    We will denote the group of Courant algebroid automorphisms by $\Aut(E)$.
\end{defn}

To better understand the group of automorphisms, we derive basic results on its infinitesimal structure.  To start, first note that any Courant algebroid $E\to M$ comes with a natural action of its sections on both sections of $E$ and smooth functions via 
  \begin{align*}
    \mc L_u v := [u, v], \qquad \mc L_u f := \rho(u) f.
  \end{align*}
  By the Leibniz rule one can thus define an action of $\mc L_u$ on arbitrary sections of tensor products of $E$ and $E^*$.  However, there is a more general notion of Lie derivative on Courant algebroids which is necessary for our purposes.  In particular, let $Z \in \mf X(E)$ be a vector field covering a vector field $\bar Z\in\mf X(M)$, which induces a one-parameter family of vector bundle automorphisms $\Phi_s$ of $E$. We can then define the \emph{generalized Lie derivative} $\mc L_Z$ by
\[\mc L_Zu:=\left.\tfrac{d}{ds}\right\rvert_{s=0}\Phi^*_s u,\qquad \mc L_Zf:=L_{\bar Z}f\qquad \forall u\in\Gamma(E),f\in C^\infty(M),\]
and again extend to tensor products via the Leibniz rule.  

\begin{defn}
    An \emph{infinitesimal Courant algebroid automorphism} is a vector field $Z\in\mf X(E)$ covering a vector field $\bar Z\in\mf X(M)$ and integrating to a vector bundle automorphism, such that for all $u,v\in\Gamma(E)$ we have
    \[\bar Z\langle u,v\rangle = \langle \mc L_Z u,v\rangle+\langle u,\mc L_Z v\rangle,\qquad \rho(\mc L_Z u)=L_{\bar Z}\rho(u),\qquad \mc L_Z[u,v]=[\mc L_Z u,v]+[u,\mc L_Z v].\]
    The Lie algebra, with the usual bracket of vector fields, of infinitesimal Courant algebroid automorphisms will be denoted $\mf{aut}(E)$.
\end{defn}

The two apparently distinct notions of Lie derivative are linked via the following proposition:

\begin{prop}[\cite{SeveraPLTDuality}]\label{prop:sections_automorphisms}
    Any section $u$ of a Courant algebroid induces a unique infinitesimal automorphism $Z_u$ which satisfies 
    \begin{enumerate}
        \item $\bar Z_u=\rho(u)$,
        \item $\mc L_{Z_u}v=[u,v]=\mc L_uv$,
        \item $[Z_u, Z_v] = Z_{[u,v]}$ and $u \mapsto Z_u$ is $\mathbb R$-linear,
        \item $Z_{\differential f} = 0$ for all $f \in C^{\infty}(M)$.
    \end{enumerate}
     
\end{prop}

\begin{proof}
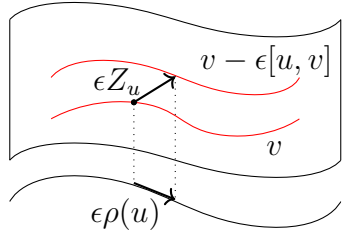
\begin{figure} 
    \begin{tikzpicture}[scale=1.1] 
    \draw plot [smooth, tension=1] coordinates {(0,0) (1,.3) (3,-.3) (4,0)};
    \draw plot [smooth, tension=1] coordinates {(0,0.5) (1,.7) (3,.3) (4,.5)}; \draw (4,.5) -- (4,2.2);
    \draw plot [smooth, tension=1] coordinates {(0,2.2) (1,2.4) (3,2) (4,2.2)}; \draw (0,2.2) -- (0,.5);
    \draw [red] plot [smooth, tension=1] coordinates {(.5,1) (1.5,1.2) (2.6,.8) (3.5,1)};
    \draw [red] plot [smooth, tension=1] coordinates {(.5,1.5) (1.3,1.7) (2.8,1.3) (3.5,1.5)};
    \fill (1.5,1.2) circle[radius=1pt]; \fill (1.5,1.2) circle[radius=1pt]; \draw[thick,->] (1.5,1.2) -- (2,1.51);
    \draw[dotted] (1.5,1.2) -- (1.5,0.2); \draw[dotted] (2,1.52) -- (2,0); \draw[thick,->] (1.5,0.225) -- (2,0.025);
    \node at (3.2,.65) {$v$}; \node at (3.1,1.7) {$v-\epsilon [u,v]$}; \node at (1.3,1.45) {$\epsilon Z_u$}; \node at (1.4,-.15) {$\epsilon \rho(u)$};
  \end{tikzpicture}
  \caption{Construction of $Z_u$}
  \label{f:figure1}
\end{figure}
We include the construction of $Z_u$ for reference (cf.\ Figure \ref{f:figure1}). For a fixed Courant algebroid $\pi\colon E\to M$, a section $u\in\Gamma(E)$ and a point $p\in E$, choose an auxiliary local section $v$ which passes through $p$, and a curve $\gamma\colon [0,1]\to M$ with $\gamma_0=\pi(p)$ and $\dot\gamma_0=\rho(u).$ The vector $Z_u$ at point $p$ is then given as the derivative of the curve
$(v-\epsilon [u,v])_{\gamma(\epsilon)}$ at zero.  One easily checks that $Z_u$ does not depend on the choice of the auxiliary section or the curve $\gamma$.  The first two properties hold by construction, and the third property is an exercise.  The fourth property holds using that $[\differential f, u] = 0$.
\end{proof}

\subsection{Courant algebroid connections}\label{subsec:cacon}

\begin{defn}
    A \emph{Courant algebroid connection} is a map $D\colon \Gamma(E)\times\Gamma(E)\to\Gamma(E)$, written as $D_u v:=D(u,v)$, such that for all $u,v\in \Gamma(E)$, $f\in C^\infty(M)$ we have
    \[D_{fu}v=fD_uv,\qquad D_u(fv)=fD_uv+(\rho(u)f)v\]
    and which preserves the pairing, i.e.
    \[\rho(u)\langle v,w\rangle=\langle D_u v,w\rangle+\langle v,D_uw\rangle.\]
Alternatively, we can describe it as a map $D\colon \Gamma(E)\to\Gamma(E^*\otimes E)$ satisfying \[D(fu)=fDu+(\differential f)\otimes u,\]
which is compatible with the pairing. Courant algebroid connections form an affine space over \[\Gamma(E\otimes \w2 E)\subset \on{Hom}(E,\on{End}(E)).\] Following the latter interpretation, for the difference of connections $A=\nabla'-\nabla$ we will use the convention
\begin{gather} \label{f:connectionconvention}
A(u,w,v)=\langle w,\nabla'_u v-\nabla_uv\rangle=\langle w,A_u v\rangle.
\end{gather}
\end{defn}

\begin{defn}
    The \emph{torsion} of a Courant algebroid connection is the section \[T_D\in\Gamma(E^*\otimes E^*\otimes E)\cong \on{Hom}(E\otimes E,E),\] defined by
    \begin{equation}\label{eq:torsion}
        T_D(u,v)=D_u v-D_v u-[u,v]+\langle Du,v\rangle.
    \end{equation}
One easily checks that this is indeed $C^\infty(M)$-multilinear.
\end{defn}

\begin{prop} Given $D$ a Courant algebroid connection, its torsion is completely antisymmetric, i.e.\ $T_D\in \w3 E^*$.
\end{prop}
\begin{proof}
  We directly calculate
  \begin{align*}
    \langle T(u,v)+T(v,u),w\rangle&=-\langle [u,v]+[v,u],w\rangle+\langle D_w u,v\rangle+\langle D_w v,u\rangle\\
    &=-\rho(w)\langle u,v\rangle+D_w\langle u,v\rangle=0,\\
    \langle T(u,v),w\rangle+\langle T(u,w),v\rangle&=\langle D_uv,w\rangle+\langle D_u w,v\rangle-\langle [u,v],w\rangle-\langle [u,w],v\rangle\\
    &=D_u\langle v,w\rangle-\rho(u)\langle v,w\rangle=0.\qedhere
  \end{align*}
\end{proof}

\begin{rmk}
    We will adhere to the standard notation for writing components of covariant derivatives in a local frame $e_\alpha$, such as
    \[D_\alpha D_{\beta}u^\gamma=(D(Du))_{\alpha\beta}^{\hphantom{\alpha\beta}\gamma}.\]
    Note that we also have
    \begin{equation}\label{eq:compare_connection_notations}
        D_\alpha D_\beta t^{\gamma\dots\delta}=(D_{e_\alpha}D_{e_\beta}t)^{\gamma\dots\delta}-(D_{D_{e_\alpha}e_\beta}t)^{\gamma\dots\delta}
    \end{equation}
    and
    \[u_\beta\,D_\gamma t\indices{^{\beta\delta\dots\epsilon}}=u^\beta D_\gamma t\indices{_\beta^{\delta\dots\epsilon}}.\]
    The operation $D_\alpha$ is a derivation in the sense that
    \[D_{\alpha}(t^{\beta\dots\gamma}u^{\delta\dots\epsilon})=(D_{\alpha}t^{\beta\dots\gamma})u^{\delta\dots\epsilon}+t^{\beta\dots\gamma}D_\alpha u^{\delta\dots\epsilon},\]
    and the same holds for $[D_\alpha,D_\beta]:=D_\alpha D_\beta-D_\beta D_\alpha$.
\end{rmk}

\subsection{Densities}

\begin{defn} \label{def:densities}
    Let $w\in \mb R$ and $r_w$ be a 1-dimensional representation of $GL(n,\mb R)$ on $\mb R$, given by $r_w(A)=|\det A|^{-w}$. For any manifold $M$ let $\on{Fr}M$ denote the frame bundle. The sections of the associated line bundle $\on{Fr}M\times_{r_w}\mb R$ are called \emph{$w$-densities}. In the particular cases of $w=1$ and $w=1/2$, we talk about \emph{densities} and \emph{half-densities}, respectively. We will denote the space of half-densities by $\ms H$ and the space of everywhere positive half densities by $\ms H^*$.  We note that a Riemannian metric $g$ induces a canonical half-density $\gs_g$. In particular on an oriented manifold $\gs_g^2$ can be naturally identified with the metric volume form.
  \end{defn}

\begin{rmk}
  We can now use the anchor $\rho\colon \Gamma(E)\to \mf X(M)$ and the ordinary Lie derivative to extend $\mc L$ to an action on  densities.  For instance, if $\kappa$ is an $\alpha$-density, we have
  \[\mc L_u \kappa= L_{\rho(u)}\kappa.\]
\end{rmk}

\subsection{Divergences}

\begin{defn} \label{d:divg}
    A \emph{divergence operator} (or simply \emph{divergence}) is an $\mb R$-linear map \[\div\colon \Gamma(E)\to C^\infty(M)\] such that for all $u\in\Gamma(E)$, $f\in C^\infty(M)$ we have
    \[\div(fu)=f\div u+\rho(u)f.\]
The set of all divergences forms an affine space over $\Gamma(E)$ and we will denote it by $\ms Div$.
\end{defn}
    
    \begin{ex}
        For $D$ a Courant algebroid connection we can construct the divergence
        \[\div_Du:=\on{Tr}D = D_\alpha u^\alpha .\]
    \end{ex}
    \begin{ex}
        If $\sigma\in\ms H^*$ we can construct the divergence
        \[\div_\sigma u:=\sigma^{-2}\mc L_{u}\sigma^2=2\sigma^{-1}L_{\rho(u)}\sigma=2L_{\rho(u)}\log\sigma.\]
    \end{ex}
    \begin{prop}[\cite{SeveraValach2}]
        Let $\div\in\ms Div$ and $\sigma\in\ms H^*$. Recall that $\div-\div_\sigma\in\Gamma(E^*)\cong \Gamma(E)$.
        Then the infinitesimal Courant algebroid automorphism \[Z_{\div}:=Z_{\div-\div_\sigma}\] is independent of $\sigma$.
    \end{prop}
    \begin{proof}
        For any $f\in C^\infty(M)$ we have $\div_{e^f\sigma}=\div_\sigma+\, 2\differential f$ and $Z_{\differential f}=0$.
    \end{proof}
    
    \begin{prop} [\cite{SeveraValach2}]\label{prop:divergence_automorphism_action}
        The following holds for any divergence, $u,v\in \Gamma(E)$, $f\in C^\infty(M)$:
        \[\langle \mc L_{Z_{\div}} u,v\rangle=\div[u,v]-\rho(u)\div v+\rho(v)\div u,\qquad \mc L_{Z_{\div}}f=\div\differential f.\]
    \end{prop}

\section{Generalized Riemannian geometry}

In this section we recall fundamental aspects of generalized Riemannian geometry.  In particular we recall the basic definition, and recall the structure of Courant isometries of generalized metrics.  We then discuss the space of Levi-Civita connections compatible with a metric and choice of divergence operator, recalling their existence and non-uniqueness.  This material is mostly not new, though we include the relevant  background for convenience.  There are important precursors in the physics literature \cite{Siegel, HullZwiebach,CSCW}, and our exposition follows more recent treatments in the mathematical literature \cite{GualtieriBranes, GF14, GF19, SeveraValach2}.

\subsection{Generalized metrics}

    \begin{defn} \label{d:genmetric}
        A \emph{generalized pseudometric} on a Courant algebroid is an endomorphism \[\gm\colon E\to E,\qquad \text{s.t.}\qquad \gm^T=\gm\quad\text{and}\quad \gm^2=\Id_E.\]
        Any generalized pseudometric induces an orthogonal eigenbundle decomposition \[E=V_+\oplus V_-,\qquad \gm|_{V_\pm}=\pm\Id_{V_\pm}.\]
    Conversely, any orthogonal decomposition of $E$ into subbundles $V_\pm$ corresponds to a generalized pseudometric.  A \emph{generalized metric} is a generalized pseudometric for which $\langle\cdot,\cdot\rangle|_{V_+}$ is positive definite and $\rho|_{V_+}\colon V_+\to TM$ is an isomorphism.  Note that the existence of a generalized metric requires that $\rho$ is at least surjective, i.e.\ that the underlying Courant algebroid is transitive.  Finally, we say that the generalized metric is \emph{strictly positive} if
    \begin{align*}
        \IP{ \gm \cdot, \cdot} > 0.
    \end{align*}
    \end{defn}

    \begin{rmk}
      Throughout this article we will assume (as part of the definition of the generalized pseudometric) that $\on{rank}V_+\neq 1$ and $\on{rank} V_-\neq 1$.  The role of this assumption is to ensure the existence of Levi-Civita connections with arbitrary divergence operator.  In all interesting examples this condition automatically holds.
    \end{rmk}
    
    \begin{rmk} \label{r:strictpositive}
      Note that $\gm$ induces in particular a fiberwise inner product $\langle \gm \slot,\slot\rangle$, which we will use to define a fiberwise norm $|\slot|_{\gm}$ on $E\otimes\dots\otimes E$.  In general, even for a generalized metric, this inner product will not necessarily be positive definite.  This is precisely the condition guaranteed by strict positivity.  Importantly, we will \emph{not} use this metric to raise and lower indices, which is consistently done via the Courant algebroid inner product $\langle \slot,\slot\rangle$.  
    \end{rmk}
    
    \begin{nota}
      We shall denote the projections onto $V_\pm$ either by subscripts $(\cdot)_\pm$ or by $p_\pm$. To simplify the formulas, we will also denote frames of $V_+$ by $e_a$, $e_b$, \dots, frames of $V_-$ by $e_{\hat a}$, $e_{\hat b}$, \dots, and frames of $E$ by $e_\alpha$, $e_\beta$, \dots.  
    \end{nota}
    
    \begin{ex}\label{ex:exact_metric}
        There is a one-to-one correspondence, for a fixed $M$, between pairs $(g,H)$ of a Riemannian metric $g$ and closed 3-form $H$, and generalized metrics on exact Courant algebroids over $M$, up to Courant automorphisms covering the identity map of the base.  This correspondence can be described as follows. 
        
        Start from an exact Courant algebroid with a chosen identification $E\cong TM\oplus T^*M$, corresponding to some $H_0\in \Omega^3_{cl}(M)$, and with a generalized metric. Since the subbundle $V_+$ is transverse to $T^*M$, it is the graph of a map $TM\to T^*M$. Seeing the latter as a section of $T^*M\otimes T^*M$, let us denote by $g$ and $B$ its symmetric and antisymmetric part, respectively. One then easily sees that $g$ and $H:=H_0+dB$ are independent of the choice of identification $E\cong TM\oplus T^*M$, and $g$ is a Riemannian metric. In particular, note that the generalized metric always singles out a preferred identification, namely the one for which $B=0$.
        
        Thus, to go in the other direction, starting with a pair $(g,H)$, we can simply take $E:=TM\oplus T^*M$ with the bracket twisted by $H$ (see Example \ref{ex:exact}) and $V_+=\mathrm{graph}(g)$. 
    \end{ex}

    \begin{defn} \label{d:Mtangent}
        For a given Courant algebroid, we let $\ms M, \ms M^+$ denote the space of all generalized pseudometrics and generalized metrics, respectively.  An elementary computation shows that the tangent space to $\ms M$ at $\gm$ can be naturally identified with symmetric endomorphisms anticommuting with $\gm$,
    \[T_{\gm}\ms M\cong \{\chi\colon E\to E\mid \chi^T=\chi\;\text{and}\; \chi\gm+\gm\chi=0\}.\]
        Note that the two conditions can be rewritten as
        \[\chi_{ab}=\chi_{\hat a\hat b}=0,\qquad \chi_{a\hat a}=\chi_{\hat a a}.\]
    \end{defn}

\begin{defn} \label{d:inducedmetricdilaton}
            Given $\gm\in\ms M^+$, define $\hat \rho\colon TM\to E$ be the inverse of $\rho|_{V_+}$ composed with the inclusion $V_+\to E$.
            We then define the \emph{induced metric} $g$ by
            \[g(X,Y):=\tfrac12\langle \hat \rho(X),\hat\rho(Y)\rangle.\]
            Note that the prefactor is chosen such that in the exact case we recover the usual Riemannian metric (see Example \ref{ex:exact_metric}). Furthermore, for any $\gm\in\ms M^+$ and $\sigma\in\ms H^*$ we define the \emph{dilaton} $\varphi\in C^\infty(M)$ by
                \[\varphi:=-\log \frac{\sigma}{\sigma_g},\quad \text{i.e.}\quad \sigma=e^{-\varphi}\sigma_g,\]
                where $\sigma_g$ is the Riemannian half-density of Definition \ref{def:densities}.
\end{defn}

\subsection{Isometries of generalized metrics}

\begin{defn}
Let $E$ be a Courant algebroid with generalized metric $\gm$.  A Courant automorphism $\Phi$ is an \emph{isometry of $\gm$} if any of the following equivalent conditions is satisfied:
\begin{enumerate}
  \item $\Phi$ preserves the subbundle $V_+$
  \item $\Phi$ preserves the subbundle $V_-$
  \item $\gm \circ\, \Phi = \Phi\circ \gm.$
\end{enumerate}
Similarly, an infinitesimal Courant automorphism $Z$ is a \emph{generalized Killing vector field} if either
\begin{enumerate}
  \item $Z$ preserves $V_+$ (equivalently, $\mc L_Z$ preserves $\Gamma(V_+)$)
  \item $Z$ preserves $V_-$ (equivalently, $\mc L_Z$ preserves $\Gamma(V_-)$)
  \item $\mc L_Z\gm=0$.
\end{enumerate}
In particular, if $\Phi_s$ is a one-parameter family of isometries with $
\Phi_0=\Id$ then $Z:=\left.\dels\right|_0\Phi_s\in\mf{aut}(E)$ is a generalized Killing vector field.
\end{defn}

\begin{defn}
        A divergence $\div$ is \emph{compatible} with a generalized pseudometric if the vector field $Z_{\div}$ is generalized Killing. Note that if $\div=\div_\sigma$ for some $\sigma\in\ms H^*$, then we have $Z_{\div}=0$. Thus all half-density-induced divergences are automatically compatible with all generalized metrics.
\end{defn}
    
    \begin{ex}[\cite{GRFbook,SeveraValach2}]
        Let $E$ be an exact Courant algebroid with a generalized metric, corresponding to a pair $(g,H)$. Let $\sigma$ be the half-density given by $g$ and $E\cong TM\oplus T^*M$ the identification provided by the generalized metric. Any divergence is then of the form
        \[\div=\div_\sigma+\,\langle X+\alpha,\slot\rangle,\]
        for $X\in\mf X(M)$ and $\alpha\in\Omega^1(M)$. It is compatible with the generalized metric iff
        \[L_Xg=0 \qquad \text{and}\qquad d\alpha=i_XH.\]
        As a particular case observe that we can take $X=0$ and $\alpha=d\varphi$.
        
        Replacing the class of half-density-induced divergences by a larger class of metric-compa\-ti\-ble divergences corresponds physically to passing from ordinary to generalized supergravity equations \cite{Arutyunov:2015mqj, Wulff:2016tju}.
    \end{ex}

\subsection{Levi-Civita connections}\label{subsec:lc}
    We now turn to a natural generalized-geometric analogue of Levi-Civita connections. It is important to stress that although these objects always exist, they are in general not unique. We start with a brief discussion of the torsion map.

    \begin{defn}
        We define the \emph{torsion map} as the vector bundle map
        \[\tau\colon E\otimes\w2E\to\w3E,\qquad \tau(A):=T_{D+A}-T_D,\]
        where $D$ is an arbitrary Courant algebroid connection. One easily sees that the definition is independent of the choice of $D$.
    \end{defn}
    \begin{lemma}
        $\tau$ equals $-3$ times the antisymmetrisation map $E\otimes\w2E\to\w3E$.
    \end{lemma}
    \begin{proof}
        $\tau(A)(u,v,w)=\langle A_uv-A_vu,w\rangle+\langle A_w u,v\rangle=A(u,w,v)+A(v,u,w)+A(w,v,u)$.
    \end{proof}    
   
    \begin{defn}
        A Courant algebroid connection $D$ is \emph{compatible} with the generalized pseudometric if $D\gm=0$, or equivalently $D_u\Gamma(V_\pm)\subset \Gamma(V_\pm)$ for every $u\in\Gamma(E)$. It is called \emph{Levi-Civita} if it furthermore has vanishing torsion.
    \end{defn}
    \begin{lemma}\label{lem:torsion_on_functions}
      If $D$ has vanishing torsion, then for any $f\in C^\infty(M)$ we have
      \begin{align}\label{eq:commut_functions}
          [D_\alpha,D_\beta]f=\rho(\langle D e_\alpha,e_\beta\rangle)f.
      \end{align}
    \end{lemma}
    \begin{proof}
      We directly calculate
      \begin{align*}
        [D_\alpha,D_\beta]f&=([D_{e_\alpha},D_{e_\beta}]-D_{D_{e_\alpha}e_\beta}+D_{D_{e_\beta}e_\alpha})f=([\rho(e_\alpha),\rho(e_\beta)]-\rho(D_{e_\alpha}e_\beta)+\rho(D_{e_\beta}e_\alpha))f\\
        &=\rho([e_\alpha,e_\beta]-D_{e_\alpha}e_\beta+D_{e_\beta}e_\alpha)f=\rho(\langle D e_\alpha,e_\beta\rangle)f.\qedhere
      \end{align*}
    \end{proof}
    \begin{rmk}
        Note that due to $\rho\circ \differential=0$ we have $\rho(\langle Du,v\rangle)=u^\alpha v^\beta\rho(\langle De_\alpha,e_\beta\rangle)$. Thus both sides of \eqref{eq:commut_functions} are in fact tensorial in $\alpha$ and $\beta$ (i.e.\ they are the $\alpha$, $\beta$ components of a tensor in $\w2E^*$).
    \end{rmk}
    
    \begin{prop}\label{prop:lcbracket} (cf.\ \cite{GF19} Lemma 3.2)
        If $D$ is Levi-Civita, then $D_{u_+}v_-=[u_+,v_-]_-$ and $D_{v_-}u_+=[v_-,u_+]_+$.
    \end{prop}
    \begin{proof}
        Vanishing torsion implies $D_{u_+}v_--D_{v_-}u_+=[u_+,v_-]$. The proposition corresponds to taking the $V_-$ and $V_+$ components of this equation, respectively.
    \end{proof}

    \begin{ex}[\cite{GualtieriBranes}]\label{ex:bismut}
        Let $E$ be an exact Courant algebroid with a generalized metric $\gm$, corresponding to a pair $(g,H)$. Let $\nabla^+$ and $\nabla^-$ be the (ordinary) metric connections on $TM$, with torsions $H$ and $-H$, respectively. Then
        \[\rho([u_-,v_+]_+)=\nabla^+_{\rho(u_-)}\rho(v_+)\qquad \text{and}\qquad \rho([u_+,v_-]_-)=\nabla^-_{\rho(u_+)}\rho(v_-).\]
    \end{ex}

    \begin{lemma}
    Let $\gm$ be a generalized pseudometric. Then the map
      \[\tau'\colon \w{3}\!\!E\to E\otimes(\w2V_+\oplus\w2V_-)\subset \on{Hom}(E,\on{End}(V_+)\oplus\on{End}(V_-))\]
      \[(\tau't)^{abc}:=-\tfrac13t^{abc},\quad (\tau't)^{\hat abc}:=-t^{\hat abc},\quad (\tau't)^{a\hat b\hat c}:=-t^{a\hat b\hat c},\quad (\tau't)^{\hat a\hat b\hat c}:=-\tfrac13t^{\hat a\hat b\hat c},\]
      or, more compactly,
      \[(\tau't)^{\alpha\beta\gamma}=-\tfrac13(t^{\alpha\beta\gamma}+\gm\indices{^\alpha_\delta}t^{\delta\epsilon[\beta}\gm\indices{^{\gamma]}_\epsilon}+t^{\alpha\delta\epsilon}\gm\indices{^\beta_\delta}\gm\indices{^\gamma_\epsilon}),\]
      satisfies $\tau\circ \tau'=\Id$, where $\tau$ is the torsion map defined above.
    \end{lemma}
    \begin{proof}
      This is straightforward, for instance:
      \[(\tau\tau't)^{abc}=-3(\tau't)^{[abc]}=t^{abc},\qquad (\tau\tau't)^{\hat abc}=-3(\tau't)^{[\hat abc]}=-\tfrac36((\tau't)^{\hat abc}-(\tau't)^{\hat acb})=t^{\hat abc}.\qedhere\]
    \end{proof}
    \begin{cor}
      For any generalized pseudometric there exists a Levi-Civita connection.
    \end{cor}
    \begin{proof} We first construct a connection compatible with $\gm$.  Choose arbitrary (ordinary) inner-product-preserving vector bundle connections  $\nabla^{V_\pm}$ on $V_\pm$, and then define 
    \[D_uv:=\nabla^{V_+}_{\rho(u)}v_++\nabla^{V_-}_{\rho(u)}v_-.\]
    If $T_D$ is the torsion of this connection, we can now form a new compatible connection $D-\tau'(T_{D})$ whose torsion is $T_{D-\tau'(T_{D})}=T_{D}-\tau\tau'(T_{D})=0$.
    \end{proof}
    Levi-Civita connections form an affine space over sections of
    \begin{equation}\label{eq:lcaffine}
      \on{ker}\tau\cap (E\otimes(\w2V_+\oplus\w2V_-)),
    \end{equation}
    i.e.\ of the bundle consisting of elements $A\in E\otimes E\otimes E$ satisfying
    \[A_{\alpha\beta\gamma}=-A_{\alpha\gamma\beta},\quad A_{\alpha a\hat b}=0,\quad A_{[\alpha\beta\gamma]}=0.\]
    We can impose a further condition on a Levi-Civita connection, which is to fix a prescribed divergence operator.  
    To achieve this, first define the map \[\kappa\colon E\otimes \w2E\to E\otimes E\otimes E\to E\] as a contraction (using the inner product) on the first two factors, i.e.
    \[\kappa(A)^\alpha:=A\indices{^\beta_\beta^\alpha}.\]
    Note that $\kappa\circ \tau'=0$. Set \[n_\pm:=\on{rank}(V_\pm).\]
    \begin{lemma}
      Let $\gm$ be a generalized pseudometric. Then
      \[\kappa'\colon E\to \on{ker}\tau\cap (E\otimes(\w2V_+\oplus\w2V_-)),\quad u\mapsto \tfrac1{n_+-1}e_a\otimes(e^a\wedge u_+)+\tfrac1{n_--1}e_{\hat a}\otimes(e^{\hat a}\wedge u_-)\]
      satisfies $\kappa\circ\kappa'=\Id$. In components we have
      \[(\kappa'u)_{abc}=\tfrac2{n_+-1}\eta_{a[b}u_{c]},\qquad (\kappa'u)_{\hat a\hat b\hat c}=\tfrac2{n_--1}\eta_{\hat a[\hat b}u_{\hat c]},\]
      or in a more unified fashion:
      \[(\kappa'u)_{\alpha\beta\gamma}=\tfrac{2-(n_++n_-)}{2(n_+-1)(n_--1)}(\eta_{\alpha[\beta}u_{\gamma]}+\gm_{\alpha[\beta}\gm\indices{_{\gamma]}_\delta}u^\delta)+\tfrac{n_+-n_-}{2(n_+-1)(n_--1)}(\gm_{\alpha[\beta}u_{\gamma]}+\eta_{\alpha[\beta}\gm_{\gamma]\delta}u^\delta).\]
    \end{lemma}
    \begin{proof}
      We have $\kappa\kappa'(u)=\tfrac1{n_+-1}(\langle e_a,e^a\rangle u_+-\langle e_a,u_+\rangle e^a)+\tfrac1{n_--1}(\langle e_{\hat a},e^{\hat a}\rangle u_--\langle e_{\hat a},u_-\rangle e^{\hat a})=u$. 
    \end{proof}
    Note that it is precisely for this result that we require the assumption $n_\pm\neq 1$.
    
    \begin{cor}[\cite{GF19} Proposition 3.3]
      Let $\gm$ be a generalized pseudometric and $\div$ an arbitrary divergence operator. Then there exists a Levi-Civita connection $D$ with $\div_D=\div$.
    \end{cor}
    \begin{proof}
      Starting from any Levi-Civita connection $D_0$, we construct the Levi-Civita connection $D:=D_0+\kappa'(\div-\div_{D_0})$ which has $\div_{D}=\div_{D_0}+\kappa\kappa'(\div-\div_{D_0})=\div$.
    \end{proof}
    
\begin{defn} \label{d:LC} The space of Levi-Civita connections with a fixed divergence $\div$ will be denoted \[LC(\gm,\div).\]  This is an 
affine space over sections of 
    \[\on{ker}\tau\cap \on{ker}\kappa\cap (E\otimes(\w2V_+\oplus\w2V_-)),\]
    i.e.\ this space consists of tensors $A\in E\otimes E\otimes E$ satisfying
    \begin{equation*}
        A_{\alpha\beta\gamma}=-A_{\alpha\gamma\beta},\quad A_{\alpha a\hat b}=0,\quad A_{[\alpha\beta\gamma]}=0,\quad A\indices{^a_{ab}}=0,\quad A\indices{^{\hat a}_{\hat a\hat b}}=0.
    \end{equation*}
\end{defn}

We end this subsection with a proposition giving explicit formulas for the Lie derivative of a generalized (pseudo)metric both in terms of the bracket, and in terms of a choice of Levi-Civita connection, which will be crucial to what follows.

    \begin{prop}\label{prop:killing}
        If $\gm\in\ms M$ then for any $u\in\Gamma(E)$ and $a_\pm\in\Gamma(V_\pm)$ we have  \[(\mc L_u\gm)a_\pm=\pm2[u,a_\pm]_\mp\]
        and so in particular $(\mc L_u\gm) V_\pm\subset V_\mp$. If furthermore $D$ is Levi-Civita then
        \[(\mc L_u\gm)_{\alpha\beta}=2(D_{[\alpha}u_{\gamma]}\gm\indices{^\gamma_\beta}+D_{[\beta}u_{\gamma]}\gm\indices{_\alpha^\gamma}),\qquad (\mc L_u\gm)_{a\hat a}=(\mc L_u\gm)_{\hat aa}=-4D_{[a}u_{\hat a]}.\]
    \end{prop}
    \begin{proof}
        For the first claim we calculate
        \[(\mc L_u\gm)a_\pm=\mc L_u(\gm a_\pm)-\gm(\mc L_u a_\pm)=(\pm \Id-\gm)\mc L_u a_\pm=\pm 2p_\mp [u,a_\pm].\]
        For the second, since $D$ is torsion-free, we have
        \[\mc L_u=D_u-Du+(Du)^t.\]
        We then obtain
        \[(\mc L_u\gm)\indices{^\alpha_\beta}=-(D_\gamma u^\alpha)\gm\indices{^\gamma_\beta}+(D_\beta u^\gamma)\gm\indices{^\alpha_\gamma}+(D^\alpha u_\gamma)\gm\indices{^\gamma_\beta}-(D^\gamma u_\beta)\gm\indices{^\alpha_\gamma}.\qedhere\]
    \end{proof}

\subsection{Differential operators}

    \begin{defn} \label{d:Laplace}
      Let $D$ be a Levi-Civita connection for a generalized pseudometric $\gm\in\ms M$. Then we define the \emph{Laplace operator} $\Delta$ by 
      \[\Delta:=\gm^{\alpha\beta}D_\alpha D_\beta\colon \Gamma(E^{\otimes k})\to\Gamma(E^{\otimes k}).\]
    Note that on functions we have
    \[\Delta f=D_\alpha(\gm^{\alpha\beta} D_\beta f)=\div_D(\gm\differential f),\]
    and so the Laplace operator on functions is determined only by the generalized pseudometric and associated divergence operator.  When further the divergence arises from a half-density $\gs$, we denote the operator by $\gD_{\gm, \gs}$.
    \end{defn}

\begin{prop} \label{p:Laplacianelliptic}
    Given $\gm$ a generalized metric and $D$ any Levi-Civita connection, the operator $\Delta$ is elliptic.
\end{prop}
\begin{proof}
    We first observe that the differential operator $D^\alpha D_\alpha$ on $\Gamma(E^{\otimes k})$ is of first order.  Note that it suffices to show this on functions, i.e.\ for $k=0$.  Indeed, taking arbitrary $\sigma\in \ms H^*$, setting $e:=\div_D-\div_\sigma$, and using $\rho\circ \differential=0$ we get
    \[D^\alpha D_\alpha f=\div_D(\mc Df)=\langle e,\differential f\rangle,\]
    giving the claim.  Given this, we have $\Delta=D^aD_a-D^{\hat a}D_{\hat a}=2D^aD_a-D^\alpha D_\alpha$, so the symbol of $\Delta$ is twice that of $D^aD_a$, which is itself positive by the definition of generalized metric.
\end{proof}

In the case of the operator $\gD_{\gm,\gs}$, we obtain an explicit formula, showing that it equals the weighted Laplace operator associated to $g$ and $\gs$.
    
        \begin{lemma} \label{l:weightedlaplacian}
          For any $\gm\in\ms M$, $f\in C^\infty(M)$, $\sigma\in\ms H^*$ and $D\in LC(\gm,\div_\sigma)$ we have \[\rho(\gm\differential f)= \N f,\qquad \Delta_{\gm,\gs} f=2\sigma^{-1}L_{\N f}\sigma=\Delta_g f-2g^{-1}(df,d\varphi),\]
          where $\N$ is the Levi-Civita connection of $g$.
        \end{lemma}
        \begin{proof}
          Using $\rho\circ \differential=0$, for any $x\in TM$ we can write
          \[g(\rho\gm \differential f,x)=2g(\rho(\differential f)_+,x)=\langle(\differential f)_+,\hat \rho x\rangle=\langle \differential f,\hat \rho x\rangle=\rho(\hat \rho x)f=xf=\langle df,x\rangle.\]
          Similarly
          \[D_\alpha (\gm^{\alpha\beta}D_\beta f)=\div_\sigma \gm\differential f=(\div_{\sigma_g}-2\langle \differential\varphi,\slot\rangle)(\gm\differential f)=\Delta_g f-2\langle \N f,d\varphi\rangle,\]
          as claimed.
        \end{proof}

\section{Generalized Riemann and Ricci tensors} \label{s:curvature}

Building on the previous section, we discuss curvature tensors associated to a generalized metric, divergence operator, and choice of Levi-Civita connection. The core notions were introduced and investigated first in \cite{Siegel} which, in the language of the present text, corresponds to the setup of exact Courant algebroids with divergences coming from half-densities. This has been later generalized to general Courant algebroids as well as the double-field theoretic context in \cite{CSCW,GF14,GF19,Jurco:2016emw,SeveraValach1,Hohm:2011si}.  We emphasize here the key point that while the Riemann curvature tensor depends on a choice of Levi-Civita connection, for a generalized metric and choice of divergence operator there is a uniquely associated Ricci curvature (\cite{GF19}, cf.\ also \cite{Siegel,CSCW}).  Turning to the scalar curvature, we note that in prior literature a definition is given by appealing to associated Dirac operators (cf.\ \cite{CSCW,GRFbook}) or explicit frames \cite{SeveraValach2}.  We give here a direct definition in terms of the generalized curvature tensor. This has a precursor in \cite{Siegel} in the exact case with half-densities.  We show in Lemma \ref{lem:scalexplicit} that this definition again only depends on the generalized metric and divergence operator.  Using these concrete definitions, we build on \cite{SeveraValach2} and compute two key formulas showing how the Ricci and scalar curvatures change under a change in divergence operator (Propositions \ref{p:Riccidivergencechange}, \ref{p:scalardivchange}).
These play a key role in the proofs of the monotonicity formulas for generalized Ricci flow.

\subsection{Riemann tensor}
    \begin{defn}\label{def:grm}
        Let $E\to M$ be a Courant algebroid and $D$ a torsion-free Courant algebroid connection. Define the \emph{Riemann tensor} as
        \[\grm_D\colon \Gamma(E)\otimes\Gamma(E)\otimes\Gamma(E)\otimes\Gamma(E)\to C^\infty(M)\]
        \[\grm_D(w,z,x,y):=\tfrac12w^\delta y^\beta(x^\alpha[D_\alpha,D_\beta]z_\delta+z^\alpha[D_\alpha,D_\delta]x_\beta-(D_\alpha x_\beta) (D^\alpha z_\delta)).\]
    \end{defn}
    \begin{prop}
        $\grm_D$ is $C^\infty(M)$-multilinear, and hence gives a tensor $\grm_D\in\Gamma(E^{\otimes 4})$.
    \end{prop}
    \begin{proof}
        The expression is clearly tensorial in $w$ and $y$ slots.
        Since $[D_\alpha,D_\beta]$ is a derivation, using Lemma \ref{lem:torsion_on_functions} we obtain
        \[\grm_D(e_\delta,z,fx,e_\beta)-f\grm_D(e_\delta,z,x,e_\beta)=\tfrac12z^\alpha x_\beta\rho(\langle De_\alpha,e_\delta\rangle)f-\tfrac12x_\beta (D^\alpha z_\delta)\rho(e_\alpha)f=0,\]
        due to $\rho(z^\alpha \langle D e_\alpha,e_\delta\rangle)=\rho(\langle Dz,e_\delta\rangle-\differential z_\delta)=(D^\alpha z_\delta)\rho(e_\alpha)+0$. Linearity in the $z$-slot follows from the symmetry of the expression when exchanging $x\leftrightarrow z$ and $y\leftrightarrow w$.
    \end{proof}
    \begin{prop}
      We have $\grm_{\alpha\beta\gamma\delta}=\grm_{[\alpha\beta]\gamma\delta}=\grm_{\alpha\beta[\gamma\delta]}=\grm_{\delta\gamma\beta\alpha}$.
    \end{prop}
    \begin{proof}
      The first equality follows from Lemma \ref{lem:torsion_on_functions}:
      \begin{align*}
        \grm(w,w,x,e_\beta)&=\tfrac12 w^\delta x^\alpha[D_\alpha,D_\beta]w_\delta-\tfrac12 w^\delta (D_\alpha x_\beta)(D^\alpha w_\delta)\\
        &=\tfrac14 x^\alpha[D_\alpha,D_\beta](w^\delta w_\delta)-\tfrac14 (D_\alpha x_\beta)D^\alpha(w^\delta w_\delta)\\
        &=\tfrac14\rho(\langle D x,e_\beta\rangle)w^2-\tfrac14\rho((D_\alpha x_\beta)e^\alpha)w^2=0.
      \end{align*}
      The symmetry $\grm_{\alpha\beta\gamma\delta}=\grm_{\delta\gamma\beta\alpha}$ follows directly from the definition. Together this gives
      \[\grm_{\alpha\beta\gamma\delta}=\grm_{\delta\gamma\beta\alpha}=-\grm_{\gamma\delta\beta\alpha}=-\grm_{\alpha\beta\delta\gamma}.\qedhere\]
    \end{proof}

\subsection{Ricci tensor}

    To our knowledge Siegel first introduced and studied this object in \cite{Siegel}.  His approach was purely local corresponds to the exact case with divergence determined by a half-density.  Much later the concept of generalized Ricci tensor was set in the global framework of exact Courant algebroids in \cite{CSCW}, giving a substantial simplification of type II supergravity. An extension to the transitive case was then studied in \cite{GF14}. In \cite{GF19} the notion was finally extended to the general Courant algebroid case, and it was noted that the structure depends only on the generalized metric and the divergence of the connection. Very shortly after the latter, \cite{Jurco:2016emw} appeared, where various other aspects of the general setup were investigated.  In \cite{SeveraValach1} another road was taken and a more restricted notion (depending on the generalized metric and an ordinary vector bundle connection) of a generalized Ricci tensor was defined and used to prove the compatibility of general Poisson--Lie T-duality \cite{Klimcik:1995ux} with 1-loop renormalisation group flow. This was later extended in \cite{SeveraValach2} to the case of equivariant Poisson--Lie T-duality \cite{Klimcik:1996np}, where the Ricci tensor was used in the form derived below in Proposition \ref{prop:grc_explicit}.

    For the sake of maximal clarity and in order to capture the most general situation, we will now define the Ricci tensor in two steps --- as a contraction of the Riemann tensor, followed by a projection.
    
    \begin{defn}[\cite{Siegel,CSCW}]
      Let $E$ be a Courant algebroid and $D$ a torsion-free Courant algebroid connection. Define first the \emph{full Ricci tensor} $\fgrc\in\Gamma(\on{End}(E))$ by 
        \begin{equation*}
          \langle v,\fgrc(u)\rangle:=\grm(e^\alpha,v,e_\alpha,u),\qquad \fgrc\indices{^\gamma_\beta}=\grm\indices{^\alpha^\gamma_{\alpha\beta}}.
        \end{equation*}
      The \emph{Ricci tensor} $\grc\in\Gamma(\on{End}(E))$ is then defined as
      \[\grc:=\gm[\gm,\fgrc]=\fgrc-\gm\fgrc\gm,\qquad \grc\indices{^\alpha_\delta}=\fgrc\indices{^\alpha_\delta}-\gm\indices{^\alpha_\beta}\fgrc\indices{^\beta_\gamma}\gm\indices{^\gamma_\delta}.\]
      Explicitly,
      \[\grc_{a\hat a}=\grc_{\hat a a}=2\fgrc_{a\hat a}=2\fgrc_{\hat aa},\qquad \grc_{ac}=\grc_{\hat a\hat c}=0.\]
      Invariantly, $\grc$ is (twice) the projection of $\fgrc$ onto $(V_+\otimes V_-)\oplus (V_-\otimes V_+)$.  Note that both $\fgrc_{\alpha\beta}$ and $\grc_{\alpha\beta}$ are symmetric. 
    \end{defn}

    The next proposition links the present approach with that of \cite{SeveraValach2}, and extends the corresponding formula from \cite{Siegel}.
    \begin{prop}\label{prop:grc_explicit} 
      For any $D\in LC(\gm,\div)$ the Ricci tensor satisfies
      \begin{align*}
        \langle v_-,\grc(u_+)\rangle&=v_-^{\hat a}[D_a,D_{\hat a}]u_+^a+u_+^{a}[D_{\hat a},D_{a}]v_-^{\hat a}\\
        &=\div\gm[v_-,u_+]-\rho(v_-)\div u_+-\rho(u_+)\div v_--2\on{Tr}_{V_+}[[\cdot, v_-]_-,u_+]_+
      \end{align*}
      If $\div$ is compatible with $\gm$ then the expression simplifies to
      \[\tfrac12\langle v_-,\grc(u_+)\rangle=v_-^{\hat a}[D_a,D_{\hat a}]u_+^a=\div[v_-,u_+]_+-\rho(v_-)\div u_+-\on{Tr}_{V_+}[[\cdot, v_-]_-,u_+]_+.\]
      In particular, $\grc$ does not depend on the choice of representative $D\in LC(\gm, \div)$. 
    \end{prop}
    \begin{proof}
    First,
    $\langle v_-,\grc(u_+)\rangle =-2\grm(e^\alpha,u_+,v_-,e_\alpha)=v_-^{\hat a}[D_a,D_{\hat a}]u_+^a+u_+^{a}[D_{\hat a},D_{a}]v_-^{\hat a}$.
      Since $\div u_+=D_\alpha u_+^\alpha=D_a u_+^a$, we have
        \begin{align*}
            v_-^{\hat a}[D_a,D_{\hat a}]u_+^a&=D_a(v_-^{\hat a}D_{\hat a}u_+^a)-(D_av_-^{\hat a})(D_{\hat a}u_+^a)-v_-^{\hat a}D_{\hat a}D_au^a_+\\
            &=D_a[v_-,u_+]^a-[e_a,v_-]^{\hat a}(D_{\hat a}u_+^a)-\rho(v_-)\div u_+\\
            &=\div[v_-,u_+]-(D_{[e_a,v_-]_-}u_+)^a-\rho(v_-)\div u_+\\
            &=\div[v_-,u_+]_+-[[e_a,v_-]_-,u_+]_+^a-\rho(v_-)\div u_+.
        \end{align*}
      Similarly,
      \begin{align*}
        u_+^a[D_{\hat a},D_a]v_-^{\hat a}&=\div[u_+,v_-]_--[[e_{\hat a},u_+]_+,v_-]_-^{\hat a}-\rho(u_+)\div v_-\\
        &=-\div[v_-,u_+]_--[[e_a,v_-]_-,u_+]_+^a-\rho(u_+)\div v_-,
      \end{align*}
      due to the cyclicity of the trace. Finally, using Proposition \ref{prop:divergence_automorphism_action}, the difference yields
      \[v_-^{\hat a}[D_a,D_{\hat a}]u_+^a-u_+^a[D_{\hat a},D_a]v_-^{\hat a}=\div[v_-,u_+]-\rho(v_-)\div u_++\rho(u_+)\div v_-=\langle \mc L_{\div} v_-,u_+\rangle,\]
      which vanishes if $\div$ is compatible with $\gm$.
    \end{proof}
    \begin{nota}
        Based on the Proposition and for typographical convenience, whenever we need to display the particular dependence of the Ricci tensor, we shall denote it either by $\grc_{\gm,\div}$ or $\grc(\gm,\div)$. In the case when $\div=\div_\sigma$, we will simply use $\grc_{\gm,\sigma}$ or $\grc(\gm,\sigma)$.
    \end{nota}
    
    \begin{prop} \label{p:Riccidivergencechange}
      $\grc_{\gm,\div+\langle e,\cdot\rangle}-\grc_{\gm,\div}=-\tfrac12\mc L_{\gm \!e}\gm$.
    \end{prop}
    \begin{proof}
        From Propositions \ref{prop:killing} and \ref{prop:grc_explicit} we have
        \begin{align*}
            \langle v_-,(\text{LHS}) u_+\rangle&=\langle e,\gm[v_-,u_+]\rangle-\rho(v_-)\langle e,u_+\rangle-\rho(u_+)\langle e,v_-\rangle\\
            &=\langle \gm \!e,[v_-,u_+]\rangle-\rho(v_-)\langle \gm\!e,u_+\rangle+\rho(u_+)\langle \gm \! e,v_-\rangle\\
            &=-\langle [v_-,\gm \!e],u_+\rangle+\rho(u_+)\langle \gm\!e,v_-\rangle=\langle [\gm\!e,v_-],u_+\rangle=-\langle v_-,[\gm\!e,u_+]\rangle\\
            &=-\tfrac12\langle v_-,(\mc L_{\gm \!e}\gm)u_+\rangle.
        \end{align*}
        The result then follows from the fact that both the LHS in the proposition and $\mc L_{\gm \!e}\gm$ are symmetric and have vanishing $V_+\otimes V_+$ and $V_-\otimes V_-$ parts.
    \end{proof}
    
    \begin{ex}
        Let $E$ be an exact Courant algebroid with a generalized metric $\gm$, corresponding to a pair $(g,H)$. Let $\div$ be the divergence w.r.t.\ the half-density given by $g$. Then
        \begin{align}\label{eq:fgrc_comp}
            \langle u_+,\fgrc(v_-)\rangle=\rc^+(\rho(u_+),\rho(v_-)),
        \end{align}
        where $\rc^+$ is the Ricci tensor for the connection $\nabla^+$ (see Example \ref{ex:bismut}). Explicitly, if $\rc$ denotes the classic Ricci tensor of $g$, we have
        \[\rc^+_{ij}=\rc_{ij}-\tfrac14H_{ikl}H_j{}^{kl}-\tfrac12 (d^*_g H)_{ij}.\]
        Similarly
        \begin{align}\label{eq:grc_comp}
            \rho\circ \grc=\rc^+\circ\rho\colon V_-\to TM,
        \end{align}
        where again the endomorphism $\rc^+$ is defined by $g(x,\rc^+(y)):=\rc^+(x,y)$. Note that when passing from \eqref{eq:fgrc_comp} and \eqref{eq:grc_comp}, the relative factor of 1/2 between the relevant components of $\grc$ and $\fgrc$ is compensated by the factor of 2 between the metric on $TM$ and $V_+$ (see Definition \ref{d:inducedmetricdilaton}).

        More generally, in the case of a nontrivial dilaton, i.e.\ for $\sigma=e^{-\varphi}\sigma_g$, we have
        \[\rho\circ \grc=(\rc^++\,i_{\nabla \varphi}H+ L_{\nabla \varphi}g)\circ \rho\colon V_-\to TM,\]
        where $\nabla$ is the Levi-Civita connection for $g$. Similarly, denoting the transpose by ${}^T$, one derives
        \begin{equation}\label{eq:rhoricciplus}
            \rho\circ \grc=-(\rc^++\,i_{\nabla \varphi}H+ L_{\nabla \varphi}g)^T\circ \rho\colon V_+\to TM.
        \end{equation}
    \end{ex}

\subsection{Scalar curvature}

    \begin{defn}
      Let $E$ be a Courant algebroid with a generalized pseudometric $\gm$ and a Levi-Civita connection $D$ with divergence $\div$. We define the \emph{scalar curvature} as
      \[\gr:=\on{Tr}(\gm\fgrc)=\gm\indices{^\beta_\alpha}\grm\indices{^{\gamma \alpha}_{\gamma\beta}}=\grm\indices{^{ab}_{ab}}-\grm\indices{^{\hat a\hat b}_{\hat a\hat b}},\]
      the final equality following since by the definition of $\grm$ one has
      \[\grm\indices{^{a\hat a}_{a\hat a}}=0.\]
    \end{defn}

    \begin{defn}
      We say that a frame $\{e_\alpha\}=\{e_a,e_{\hat a}\}$ adapted to the subbundle decomposition $E=V_+\oplus V_-$ is \emph{admissible} if $\langle e_\alpha,e_\beta\rangle$ are constant functions. For such a frame it follows that the \emph{structure coefficients}
    \[c_{\alpha\beta\gamma}:=\langle [e_\alpha,e_\beta],e_\gamma\rangle\]
    are then completely antisymmetric.
    \end{defn}

    The next lemma again provides a link to \cite{SeveraValach2}, and extends the corresponding formula from \cite{Siegel}.
    \begin{lemma}\label{lem:scalexplicit}
      In an admissible frame we have
      \begin{align*}
      \gr&=-(\div e_\alpha)(\div\gm \!e^\alpha)-2\rho(e_\alpha)\div\gm \!e^\alpha+\tfrac14\gm\indices{^\alpha_\delta}c_{\alpha\beta\gamma}c^{\delta\beta\gamma}-\tfrac1{12}\gm\indices{^\alpha_\delta}\gm\indices{^\beta_\epsilon}\gm\indices{^\gamma_\zeta}c_{\alpha\beta\gamma}c^{\delta\epsilon\zeta}.
      \end{align*}
      In particular $\gr$ only depends on the connection via its divergence.
    \end{lemma}
    \begin{proof}
      Fix an admissible frame, and denote the Christoffel symbols \[\Gamma\indices{_\alpha^\gamma_\beta}:=(D_{e_\alpha}e_\beta)^\gamma,\] which satisfy $\Gamma\indices{_{\alpha\beta\gamma}}=-\Gamma\indices{_{\alpha\gamma\beta}}$.
      The vanishing of torsion, compatibility with $\gm$, and Proposition \ref{prop:lcbracket} give
      \[\Gamma_{[\alpha\beta\gamma]}=-\tfrac13c_{\alpha\beta\gamma},\qquad \Gamma_{\alpha a\hat a}=\Gamma_{\alpha \hat aa}=0,\qquad \Gamma_{\hat a bc}=-c_{\hat a  bc},\quad \Gamma_{a \hat b\hat c}=-c_{a\hat b\hat c},\]
      respectively. In particular
      \[\Gamma\indices{_a^a_b}=-\Gamma\indices{^a_{ba}}-\Gamma\indices{_{ba}^a}-c\indices{_a^a_b}=-\Gamma\indices{^a_{ba}}=\div e_b\]
      and \[\Gamma_{abc}\Gamma^{cba}+\tfrac12\Gamma_{cab}\Gamma^{cba}=\tfrac32\Gamma_{[abc]}\Gamma^{cba}=-\tfrac32\Gamma_{[abc]}\Gamma^{[abc]}=-\tfrac16c_{abc}c^{abc}.\]
      We then directly calculate
        \begin{align*}
            \grm\indices{^{ab}_{ab}}&=\tfrac12\eta^{\delta a}\delta^\beta_b((e_a)^\alpha[D_\alpha,D_\beta](e^b)_\delta+(e^b)^\alpha[D_\alpha,D_\delta](e_a)_\beta-(D_\alpha(e_a)_\beta)(D^\alpha(e^b)_\delta))\\
            &=[D_a,D_b](e^b)^a-\tfrac12\Gamma_{\alpha ab}\Gamma^{\alpha ba}=D_aD_b(e^b)^a-D_bD_a(e^b)^a-\tfrac12\Gamma_{c ab}\Gamma^{c ba}-\tfrac12\Gamma_{\hat c ab}\Gamma^{\hat c ba}\\
            &=(\div(D_{e_b}e^b)-(D_{D_{e_a}e_b}e^b)^a)-\rho(e_b)\div(e^b)-\tfrac12\Gamma_{c ab}\Gamma^{c ba}-\tfrac12\Gamma_{\hat c ab}\Gamma^{\hat c ba}\\
            &=\div(\Gamma\indices{_b_a^b}e^a)-\Gamma_{abc}\Gamma^{cba}-\rho(e_a)\div(e^a)-\tfrac12\Gamma_{c ab}\Gamma^{c ba}+\tfrac12\Gamma_{\hat c ab}\Gamma^{\hat c ab}\\
            &=-(\div e^a)(\div e_a)-2\rho(e_a)\div(e^a)+\tfrac16c_{abc}c^{abc}+\tfrac12c_{\hat cab}c^{\hat cab}.
        \end{align*}
        Now, $\grm\indices{^{\hat a\hat b}_{\hat a\hat b}}$ is obtained by simply exchanging $V_+$ and $V_-$, and so
        \[\grm\indices{^{\hat a\hat b}_{\hat a\hat b}}=-(\div e^{\hat a})(\div e_{\hat a})-2\rho(e_{\hat a})\div(e^{\hat a})+\tfrac16c_{\hat a\hat b\hat c}c^{\hat a\hat b\hat c}+\tfrac12c_{c\hat a\hat b}c^{c\hat a\hat b}.\]
      The statement of the Lemma then follows from
      \[\tfrac16c_{abc}c^{abc}+\tfrac12c_{\hat abc}c^{\hat abc}-\tfrac12c_{a\hat b\hat c}c^{a\hat b\hat c}-\tfrac16c_{\hat a\hat b\hat c}c^{\hat a\hat b\hat c}=\tfrac14\gm\indices{^\alpha_\delta}c_{\alpha\beta\gamma}c^{\delta\beta\gamma}-\tfrac1{12}\gm\indices{^\alpha_\delta}\gm\indices{^\beta_\epsilon}\gm\indices{^\gamma_\zeta}c_{\alpha\beta\gamma}c^{\delta\epsilon\zeta}.\qedhere\]
    \end{proof}
    \begin{nota}
        Similarly to the Ricci tensor, when we need to display the concrete dependencies, we shall write $
        \gr_{\gm,\div}$ or $\gr(\gm,\div)$, and in the case $
        \div=\div_\sigma$ we will use $\gr_{\gm,\sigma}$ or $\gr(\gm,\sigma)$.
    \end{nota}    
    \begin{prop} \label{p:scalardivchange}
        $\gr_{\gm,\div+\langle e,\slot\rangle}-\gr_{\gm,\div}=-|e|_{\gm}^2-2\div\gm \!e$.
    \end{prop}
    \begin{proof}
        Since in an admissible frame $\gm$ is constant, we have
        \begin{align*}
            \gr_{\gm,\div+\langle e,\slot\rangle}-\gr_{\gm,\div}=&\ -\langle e,e_\alpha\rangle\langle e,\gm \!e^\alpha\rangle-2\langle e,e_\alpha\rangle \div\gm \!e^\alpha-2\rho(\gm\!e_\alpha)\langle e,e^\alpha\rangle\\
            =&\ -\langle \gm\! e,e\rangle-2\div\gm\!e,
        \end{align*}
        as claimed.
    \end{proof}

    \begin{ex}[\cite{SeveraValach2}]
      Fix a generalized metric $\gm$ on an exact Courant algebroid, corresponding to a pair $(g,H)$, and $\varphi\in C^\infty(M)$. Let $\sigma_g$ be the metric half-density and take $\sigma:=e^{-\varphi}\sigma_g$. Then
      \begin{equation}\label{eq:curvexplicit}
          \gr_{\gm,\sigma}=R-\tfrac1{12}\brs{H}^2_g-4e^\varphi\Delta_ge^{-\varphi}.
      \end{equation}
    \end{ex}

    \begin{rmk}\label{rk:laplace}
        Let $\slashed D$ be the generating Dirac operator \cite{AXu, Severa} associated to the Courant algebroid. Recall that $\slashed D^2$ is simply multiplication by a certain function on $M$, which only depends on the Courant algebroid itself. Notably, in the exact case $\slashed D^2=0$.  This operator determines the scalar curvature of the ``maximal'' pseudometric $\mc G=\Id$ following \cite{SeveraValach2}.
        Let $\div\in\ms Div$, $\sigma\in\ms H^*$, and set $e:=\div-\div_{\sigma}$. Then
        \[\gr_{\Id,\div}=-8\slashed D^2-\langle e,e\rangle-2\div_{\sigma}e.\]
        Furthermore, following \cite{SeveraValach2}, any generalized pseudometric induces an associated second-order differential operator $\check\Delta_{\gm}$ on $\ms H$, determined uniquely by the condition
    \[\sigma^{-1}\check\Delta_{\gm}\sigma=-\tfrac12\gr_{\gm,\sigma},\qquad \forall \sigma\in\ms H^*.\]
    It is self-adjoint and relates to the generating Dirac operator by $\check\Delta_{\text{id}}=4\slashed D^2$. (Note that $\check\Delta$ is shifted by $-4\slashed D^2$ w.r.t.\ the operator from \cite{SeveraValach2}.) It relates to $\Delta$ by
    \[\sigma^{-1}\check\Delta_{\gm}(f\sigma)=2\Delta_{\gm,\sigma}f-\tfrac12f\gr_{\gm,\sigma}.\]
\end{rmk}

\section{Variational formulas} \label{s:variations}

In this section we compute variational formulas for the various curvature quantities introduced in \S \ref{s:curvature}.  
We first fix notation that we will use throughout the paper.  Let $\gm_s$, $\divg_s$, $\gs_s$, $\varphi_s$, $D_s$ denote families of generalized pseudometrics, divergence operators, half-densities, dilatons, and connections, and set
\begin{gather} \label{f:variations}
    \frac{\del}{\del s} \gm = \chi, \quad \frac{\del}{\del s} \divg = \varepsilon, \quad \dels \gs = \nu, \quad \dels \varphi = \psi, \quad \frac{\del}{\del s} D = A.
\end{gather}
The computations to follow conceal a certain subtlety: the linearization of classical curvature quantities are typically expressed as second-order differential operators using the Levi-Civita connection.  In light of the nonuniqueness of Levi-Civita connections, some care is required.  Our first step in Proposition \ref{prop:defconnection} is to show that for a one-parameter family of pairs $(\gm_s, \divg_s)$ and a choice of Levi-Civita connection for $(\gm_0, \divg_0)$, we obtain a canonical one-parameter family of Levi-Civita connections along the path $(\gm_s, \divg_s)$ which agrees with the initial choice.  Given this family, we proceed to compute variational formulas for the full curvature tensor, and the Ricci and scalar curvatures.

\subsection{Canonical Levi-Civita connection along families}

  To begin, recall that in Subsections \ref{subsec:cacon} and \ref{subsec:lc} we introduced the maps
  \[\tau\colon E\otimes \w2E\to\w3E,\qquad \kappa\colon E\otimes E\otimes E\to E,\]
  given by $-3$ times antisymmetrisation, and contraction on the first two factors, respectively, as well as the maps
  \[\tau'\colon \w3\!\!E\to E\otimes (\w2V_+\oplus\w2V_-),\qquad \kappa'\colon E\to E\otimes(\w2V_+\oplus\w2V_-),\]
  satisfying \[\tau\circ\tau'=\Id,\qquad \kappa\circ\kappa'=\Id,\qquad \kappa\circ\tau'=0,\qquad \tau\circ\kappa'=0.\]
    
    \begin{prop}\label{prop:defconnection}
       Fix $(\gm_s, \divg_s)$ a one-parameter family of generalized pseudometrics and divergences and $D_0\in LC(\gm_0,\div_0)$.  Then the solution $D_s$ of the ordinary differential equation
       \[\frac{\del}{\del s} D=\tfrac14(1-\tau'\circ \tau-\kappa'\circ \kappa)\left(D\left[\gm,\chi\right]\right)+\kappa'(\varepsilon)\]
       with initial data $D_0$ satisfies $D_s \in LC(\gm_s, \divg_s)$ for all $s$.
    \end{prop}
    \begin{proof}
        For $D_s$ to remain compatible with $\gm_s$ requires that for any $u\in \Gamma(E)$ one has
        \[0= \frac{\del}{\del s} (D_u \gm) = D_u\chi +[A_u,\gm].\]
        The vanishing of torsion and prescribed divergence impose
        \[\tau(A)=0,\qquad \kappa(A)= \varepsilon.\]
        To see that
        \[A=\tfrac14(1-\tau'\circ \tau-\kappa'\circ \kappa)(D[\gm,\chi])+\kappa'(\varepsilon)\]
        satisfies these three conditions, we use the fact that 
        $[\alpha_u, \gm] = 0$ for $\alpha \in \on{im}(\tau')$ or $\alpha \in \on{im}(\kappa')$
        to get
        \begin{align*}[A_u,\gm]&=\tfrac14[D_u[\gm,\chi],\gm]=\tfrac14D_u[[\gm,\chi],\gm]=-D_u\chi,\\
        \tau(A)&=\tfrac14(\tau-\tau\circ\tau'\circ \tau)(D[\gm,\chi])=0, \\
        \kappa(A)&=\tfrac14(\kappa-\kappa\circ\kappa'\circ \kappa)(D[\gm,\chi])+\kappa\circ\kappa'(\varepsilon)=\varepsilon.\qedhere
        \end{align*}
    \end{proof}
    
    \begin{lemma}\label{lem:eplicit} Given $(\gm_s, \divg_s)$ a one-parameter family of generalized pseudometrics and divergences and $D_0\in LC(\gm_0,\div_0)$, let $D_s$ denote the family obtained from Proposition \ref{prop:defconnection}.  Then
        \[(n_+-1)A_{abc}=\eta_{a[b}D^{\hat a}\chi\indices{_{c]\hat a}}+2\eta_{a[b}\varepsilon_{c]},\qquad (n_--1)A_{\hat a\hat b\hat c}=-\eta_{\hat a [\hat b}D^{a}\chi_{\hat c]a}+2\eta_{\hat a[\hat b}\varepsilon_{\hat c]},\]
        \[A_{ab\hat c}=\tfrac12 D_a\chi_{b\hat c},\quad A_{\hat a \hat b c}=-\tfrac12 D_{\hat a}\chi_{\hat bc},\quad A_{\hat a b\hat c}=\tfrac12 D_{\hat a}\chi_{b\hat c},\quad A_{a\hat bc}=-\tfrac12 D_a\chi_{\hat bc},\]
        \[ A_{\hat abc}=-D_{[b}\chi_{c]\hat a},\quad A_{a\hat b\hat c}=D_{[\hat b}\chi_{\hat c]a}.\]
    \end{lemma}
    \begin{proof}
        First, we have $(D[\gm,\chi])_{\alpha\beta\gamma}=2\gm_{\beta\delta}D_\alpha\chi\indices{^\delta_\gamma}$, which has vanishing $\alpha ab$ and $\alpha \hat a\hat b$ components. Then we calculate
        \begin{align*}
            A_{ab\hat c}&=\tfrac14(D[\gm,\chi])_{ab\hat c}=\tfrac12D_a\chi_{b\hat c},\\
            A_{\hat abc}&=-\tfrac14(\tau'\tau D[\gm,\chi])_{\hat abc}=\tfrac14(-3)\tfrac13(D[\gm,\chi]_{c\hat ab}-D[\gm,\chi]_{b\hat ac})=\tfrac12(D_b\chi_{c\hat a}-D_c\chi_{b\hat a}),\\
            A_{abc}&=-\tfrac14(\kappa'\kappa D[\gm,\chi])_{abc}+(\kappa'\varepsilon)_{abc}=\tfrac1{n_+-1}(\eta_{a[b}D^{\hat a}\chi\indices{_{c]\hat a}}+2\eta_{a[b}\varepsilon_{c]}).
        \end{align*}
        The remaining components follow analogously.
    \end{proof}

\subsection{Variation of curvature tensors}
    The following variation of the Riemann tensor extends the exact-case result of \cite{Hohm:2011si}, derived in the double-field theoretic context.
  \begin{prop}
    Let $D_s$ be a family of torsion-free connections. Then
    \[\dels \grm_{\delta\gamma\alpha\beta}=D_{[\alpha}A_{\beta]\delta\gamma}+D_{[\delta}A_{\gamma]\alpha\beta}.\]
  \end{prop}
  \begin{proof}
    We start by calculating
    \begin{align*}
      \dels (D_\alpha D_\beta u^\gamma)=&\ D_\alpha (\tdels D)_\beta u^\gamma+ (\tdels D)_\alpha D_\beta u^\gamma\\
=&\ D_\alpha(A\indices{_\beta^\gamma_\delta}u^\delta)+A\indices{_\alpha_\beta^\delta}D_\delta u^\gamma+A\indices{_\alpha^\gamma_\delta}D_\beta u^\delta\\
      =&\ u^\delta D_\alpha A\indices{_\beta^\gamma_\delta}+A\indices{_\alpha_\beta^\delta}D_\delta u^\gamma+2A\indices{_{(\alpha|}^\gamma_\delta}D_{|\beta)} u^\delta.
    \end{align*}
    Now we can write the generalized Riemann tensor as
    \[\grm(e_\delta,z,x,e_\beta)=x^\alpha D_{[\alpha}D_{\beta]}z_\delta-\tfrac14(D_\gamma x_\beta)(D^\gamma z_\delta)+(\beta\leftrightarrow \delta \; \& \; z \leftrightarrow x),\]
    and so
    \begin{align*}
      \dels{\grm}(e_\delta,z,x,e_\beta)&=x^\alpha \tdels (D_{[\alpha}D_{\beta]})z_\delta-\tfrac12((\tdels D)_\gamma x_\beta)(D^\gamma z_\delta)+(\beta\leftrightarrow \delta \; \& \; z \leftrightarrow x)\\
      &=x^\alpha z^\gamma D_{[\alpha} A\indices{_{\beta]}_\delta_\gamma}+x^\alpha A\indices{_{[\alpha}_{\beta]}^\gamma}D_\gamma z_\delta-\tfrac12 A\indices{_\gamma_\beta^\alpha}x_\alpha D^\gamma z_\delta+(\beta\leftrightarrow \delta \; \& \; z \leftrightarrow x)\\
      &=x^\alpha z^\gamma D_{[\alpha} A\indices{_{\beta]}_\delta_\gamma}+x^\alpha (D^\gamma z_\delta)(A_{[\alpha\beta]\gamma}-\tfrac12 A_{\gamma\beta\alpha})+(\beta\leftrightarrow \delta \; \& \; z \leftrightarrow x)\\
      &=x^\alpha z^\gamma D_{[\alpha} A\indices{_{\beta]}_\delta_\gamma}+\tfrac32 x^\alpha (D^\gamma z_\delta)A_{[\alpha\beta\gamma]}+(\beta\leftrightarrow \delta \; \& \; z \leftrightarrow x)\\
      &=x^\alpha z^\gamma D_{[\alpha} A\indices{_{\beta]}_\delta_\gamma}+(\beta\leftrightarrow \delta \; \& \; z \leftrightarrow x),
    \end{align*}
    as claimed.
  \end{proof}

  \begin{prop}\label{prop:variation}
Given $(\gm_s, \divg_s)$ a one-parameter family of generalized pseudometrics and divergences, fix $D_s\in LC(\gm_s,\div_s)$.  Then
  \begin{align}\label{eq:grc_var}
    \dels\grc&=-\tfrac12\Delta \chi - \tfrac12 \mc L_{\on{Tr} D \chi+\gm \varepsilon}\gm+\mc A(\chi)+\chi\gm\grc+\gm[\chi,\fgrc],\\
    \dels\gr&=\gm\indices{^\alpha_\gamma} \gm\indices{^\beta_\delta} D_\alpha D_\beta \chi^{\gamma\delta}-2\gm^{\alpha\beta}D_\alpha \varepsilon_\beta+\tfrac12\chi^{\alpha\beta}\grc_{\alpha\beta}.
  \end{align}
    where $(\on{Tr} D \chi)^\beta:=D_\alpha \chi^{\alpha\beta}$ and $\mc A(\chi)$ is the following first order differential operator in $\chi$:
    \[\mc A(\chi)_{\alpha\beta}=\tfrac12(\gm\indices{^\gamma_\beta}[D_\delta,D_\gamma]\chi\indices{_\alpha^\delta}-\gm\indices{^\gamma_\beta}[D_\delta,D_\alpha]\chi\indices{_\gamma^\delta})+(\alpha\leftrightarrow\beta).\]
  \end{prop}
  
  \begin{proof} Note that as both Ricci and scalar only depend on $\gm$ and $\divg$, we can choose any family of Levi-Civita connections to compute the variation.  Thus for the data $(\gm_s, \divg_s)$ and given $D_0$, let $D_s$ denote the family obtained from Proposition \ref{prop:defconnection}.
  We use the fact that $\chi_{ac}=\chi_{\hat a\hat c}=0$, together with Lemma \ref{lem:eplicit}, and its consequence
    \begin{align}\label{eq:variation_trace}
        A\indices{_a^{ca}}&=-\tfrac12D_{\hat a}\chi^{c\hat a}-\varepsilon^c,\quad A\indices{_{\hat a}^{\hat c\hat a}}=\tfrac12D_a\chi^{\hat ca}-\varepsilon^{\hat c},\quad A\indices{_a^{\hat aa}}=-\tfrac12D_a\chi^{\hat a a},\quad A\indices{_{\hat a}^{a\hat a}}=\tfrac12D_{\hat a}\chi^{a\hat a},
    \end{align}
    to first calculate:
    \begin{align*}
      \tdels\fgrc_{a\hat a}&=\tdels\grm\indices{^\beta_{a\beta \hat a}}=D_{[\beta}A\indices{_{\hat a]}^\beta_a}+D_{[\beta}A\indices{_{a]}^\beta_{\hat a}}=D_{\beta}A\indices{_{(a}^\beta_{\hat a)}}-D_{(a|}A\indices{_{\beta}^\beta_{|\hat a)}}\\
      &=\tfrac12(D_cA\indices{_{a}^c_{\hat a}}+D_{\hat c}A\indices{_{a}^{\hat c}_{\hat a}}+D_cA\indices{_{\hat a}^c_a}+D_{\hat c}A\indices{_{\hat a}^{\hat c}_a})-D_{(a}\varepsilon_{\hat a)}\\
      &=\tfrac12(\tfrac12D_cD_a\chi\indices{^c_{\hat a}}+D^{\hat c}D_{[\hat c}\chi_{\hat a]a}-D^cD_{[c}\chi_{a]\hat a}-\tfrac12D_{\hat c}D_{\hat a}\chi\indices{^{\hat c}_a})-D_{(a}\varepsilon_{\hat a)}\\
      &=\tfrac12(-\tfrac12 D^cD_c\chi_{a\hat a}+\tfrac12 D^{\hat c}D_{\hat c}\chi_{a\hat a}+D_cD_a\chi\indices{^c_{\hat a}}-D_{\hat c}D_{\hat a}\chi\indices{_a^{\hat c}})-D_{(a}\varepsilon_{\hat a)}.
    \end{align*}
    Furthermore observe 
    \begin{align*}
      \tdels\grc&=\tdels (\gm[\gm,\fgrc])= \chi\gm\grc+\gm\chi\fgrc-\gm\fgrc\chi+\gm[\gm,\tdels{\fgrc}]
    \end{align*}
    In components this reads
    \begin{align*}
      \tdels\grc_{a\hat a}&=\chi_{a\hat c}\fgrc\indices{^{\hat c}_{\hat a}}-\fgrc_{ac}\chi\indices{^c_{\hat a}}+2\tdels{\fgrc}_{a\hat a}\\
      \tdels\grc_{ac}&=-\chi_{a\hat a}\grc\indices{^{\hat a}_c}+\chi_{a\hat a}\fgrc\indices{^{\hat a}_c}-\fgrc_{a\hat a}\chi\indices{^{\hat a}_c}=-\chi_{\hat a(a}\grc\indices{_{c)}^{\hat a}}\\
      \tdels\grc_{\hat a \hat c}&=\chi_{\hat aa}\grc\indices{^a_{\hat c}}-\chi_{\hat aa}\fgrc\indices{^a_{\hat c}}+\fgrc_{\hat aa}\chi\indices{^a_{\hat c}}=\chi_{a(\hat a}\grc\indices{_{\hat c)}^a}.
    \end{align*}
    This implies
    \begin{align*}
    \dels \grc_{a\hat a} &=-\tfrac12\gm^{\alpha\beta}D_\alpha D_\beta\chi_{a\hat a}+D_cD_a\chi\indices{^c_{\hat a}}-D_{\hat c}D_{\hat a}\chi\indices{_a^{\hat c}}-2D_{(a}\varepsilon_{\hat a)}+\chi_{a\hat c}\fgrc\indices{^{\hat c}_{\hat a}}-\fgrc_{ac}\chi\indices{^c_{\hat a}},\\
    \dels\grc_{ac}&=-\chi_{\hat a(a}\grc\indices{_{c)}^{\hat a}},\qquad
    \dels\grc_{\hat a \hat c}=\chi_{a(\hat a}\grc\indices{_{\hat c)}^a},
    \end{align*}
    where the hatted/unhatted indices refer to $\gm_s$.  To further simplify we note from the definition of $\mc A(\chi)$ that
    \[\mc A(\chi)_{a\hat a}=[D_c,D_a]\chi\indices{_{\hat a}^c}-[D_{\hat c},D_{\hat a}]\chi\indices{_a^{\hat c}},\qquad \mc A(\chi)_{ac}=\mc A(\chi)_{\hat a\hat c}=0.\]
    Using Proposition \ref{prop:killing}, we then have
    \begin{align*}
        D_cD_a\chi\indices{_{\hat a}^c}-D_{\hat c}D_{\hat a}\chi\indices{_a^{\hat c}}&=D_aD_c\chi\indices{_{\hat a}^c}-D_{\hat a}D_{\hat c}\chi\indices{_a^{\hat c}}+\mc A(\chi)_{a\hat a}\\
        &= 2D_{[a|}D_{\alpha}\chi\indices{_{|\hat a]}^\alpha}+\mc A(\chi)_{a\hat a}=-\tfrac12(\mc L_{\on{Tr} D \chi}\gm)_{a\hat a}+\mc A(\chi)_{a\hat a}\\
        -2D_{(a}\varepsilon_{\hat a)}&=2D_{[a}(\gm \varepsilon)_{\hat a]}=-\tfrac12(\mc L_{\gm\!\varepsilon}\gm)_{a\hat a}.
    \end{align*}
    Finally, the rest follows from
    \begin{align*}
        (\chi\gm\grc)_{ac}=-\chi_{a\hat a}\grc\indices{^{\hat a}_c},\qquad (\chi\gm\grc)_{\hat a \hat c}=\chi_{\hat aa}\grc\indices{^a_{\hat c}},\qquad (\chi\gm\grc)_{a\hat a}=(\chi\gm\grc)_{\hat aa}=0,
    \end{align*}
    and
    \begin{align*}
        (\gm[\chi,\fgrc])_{a\hat a}=(\gm[\chi,\fgrc])_{\hat aa}&=\chi_{a\hat c}\fgrc\indices{^{\hat c}_{\hat a}}-\fgrc_{ac}\chi\indices{^c_{\hat a}},\\
        (\gm[\chi,\fgrc])_{ac}&=\chi_{a\hat a}\fgrc\indices{^{\hat a}_c}-\fgrc_{a\hat a}\chi\indices{^{\hat a}_c}=\tfrac12(\chi_{a\hat a}\grc\indices{^{\hat a}_c}-\grc_{a\hat a}\chi\indices{^{\hat a}_c}),\\
        (\gm[\chi,\fgrc])_{\hat a\hat c}&=-\chi_{\hat aa}\fgrc\indices{^a_{\hat c}}+\fgrc_{\hat aa}\chi\indices{^a_{\hat c}}=\tfrac12(-\chi_{\hat aa}\grc\indices{^a_{\hat c}}+\grc_{\hat aa}\chi\indices{^a_{\hat c}}).
    \end{align*}
    Similarly, for the scalar curvature we have
        \begin{align*}
      \tdels\gr&=\tdels (\gm^{\alpha\beta}\grm\indices{^\gamma_{\alpha\gamma\beta}})=\chi^{\alpha\beta}\fgrc_{\alpha\beta}+\gm\indices{_\alpha^\beta}(D_{[\gamma}A\indices{_{\beta]}^{\gamma\alpha}}+D^{[\gamma}A\indices{^{\alpha]}_{\gamma\beta}})\\
      &=\tfrac12\chi^{\alpha\beta}\grc_{\alpha\beta}+\gm^{\alpha\beta}(D_\gamma A\indices{_{\alpha}^\gamma_\beta}-D_{\alpha}A\indices{_{\gamma}^\gamma_\beta})\\
      &=\tfrac12\chi^{\alpha\beta}\grc_{\alpha\beta}+D_c A\indices{_a^{ca}}+D_{\hat a} A\indices{_a^{\hat aa}}-D_a A\indices{_{\hat a}^{a\hat a}}-D_{\hat c} A\indices{_{\hat a}^{\hat c\hat a}}-\gm^{\alpha\beta}D_{(\alpha}\varepsilon_{\beta)}\\
      &=\tfrac12\chi^{\alpha\beta}\grc_{\alpha\beta}+(-\tfrac12D_cD_{\hat a}\chi^{c\hat a}-D_c\varepsilon^c)-\tfrac12D_{\hat a}D_a\chi^{\hat a a}-\tfrac12D_aD_{\hat a}\chi^{\hat aa}\\
      &\ \qquad -(\tfrac12D_{\hat c}D_a\chi^{\hat ca}-D_{\hat c}\varepsilon^{\hat c})-\gm^{\alpha\beta}D_{\alpha}\varepsilon_{\beta}\\
      &=\tfrac12\chi^{\alpha\beta}\grc_{\alpha\beta}-D_{\hat a}D_a\chi^{\hat a a}-D_aD_{\hat a}\chi^{a\hat a}-2\gm^{\alpha\beta}D_{\alpha}\varepsilon_{\beta},
    \end{align*}
    as claimed.
\end{proof}

\section{Einstein-Hilbert functional}

With the definition of scalar curvature from the previous section, we can now define the generalized Einstein-Hilbert action.  We give two definitions, one depending only on a generalized metric and choice of half-density $\gs$, where the relevant divergence operator is that induced by $\gs$, and one which allows for an arbitrary divergence operator.  We then compute the Euler-Lagrange equations for this action, building on the results of \cite{SeveraValach2}. The variation is expressed in terms of a particular choice of mixed-signature $L^2$ inner product on $\ms M \times \ms H^*$.  As explained in Remark \ref{r:innerproduct}, this choice is justified on physical grounds, and furthermore leads to a number of remarkable analytic properties as detailed in the sequel.  We end this section by deriving a Bianchi identity for generalized Ricci and scalar curvature.

    \begin{defn}
        Let $E$ be a Courant algebroid over a compact base $M$. The \emph{Einstein--Hilbert functional} is a function $\eh\colon \ms M\times\ms H^* \to \mathbb R$ defined by
        \[\eh(\gm,\sigma)=\int_M \gr_{\gm, {\gs}} \gs^2.\] 
    \end{defn}
    \begin{rmk} \label{r:EHLaplace}
        Using the operator defined in Remark \ref{rk:laplace}, the Einstein--Hilbert functional takes the particularly elegant and symmetric form \[\eh(\gm,\sigma)=-2\int_M\sigma \check\Delta_{\gm}\sigma.\]
    \end{rmk}

\subsection{Euler-Lagrange equations}

Here we use the variational formulas of \S \ref{s:variations} to compute the Euler-Lagrange equations for $\eh$.  We will furthermore interpret this answer in terms of a specific choice of mixed-signature inner product on the space $\ms M \times \ms H^*$ of generalized metrics and half-densities. 

\begin{defn} \label{d:L2metric} Let $E$ be a Courant algebroid over a compact base $M$.  Define a non-degenerate pairing $(\slot,\slot)$ on $\ms M\times\ms H^*$ by
    \[(\chi,\chi)_{\gm,\sigma}=\tfrac18\int_M|\chi|^2_{\gm}\,\sigma^2=-\tfrac18\int_M\langle \chi,\chi\rangle\sigma^2,\qquad (\nu,\nu)_{\gm,\sigma}=-2\int_M\nu^2.\]
\end{defn}

    \begin{rmk} \label{r:innerproduct}
        Note that the signs are chosen so that in the strictly positive case the pairing on the variation of generalized metrics is positive definite. The choice of constants in this norm correspond physically to a particular choice of the renormalisation scheme --- namely dimensional regularisation --- in the calculation of the renormalisation group flow (see \cite{Polchinski}) of the string nonlinear sigma model. 
        As we will detail below, this choice yields many key analytic properties as well, in particular that the evolution of the dilaton under the associated gradient flow (generalized Ricci flow) contains no second order derivative terms in the metric (see Lemma \ref{l:dillinearization}).
    \end{rmk}

    \begin{prop} \label{p:EHvariation}
      For a one-parameter family $(\gm_s, \sigma_s)$ we have
      \begin{align*}
        \dels \eh (\gm, \sigma) = \int_M \tfrac12\chi^{\alpha\beta} \grc_{\alpha\beta}\sigma^2+2\gr\sigma\nu=-4(\chi,\grc)_{\gm,\sigma}-(\nu,\gr\sigma)_{\gm,\sigma}.
      \end{align*}
    \end{prop}
    \begin{proof} 
      It follows immediately from Proposition \ref{prop:variation} and the fact that \[\int_M (\div_{\sigma} e) \sigma^2 = \int_M L_{\rho(e)}\sigma^2=0.\qedhere\]
    \end{proof}

\begin{rmk} Note that critical points of $\eh$ have vanishing Ricci and scalar curvatures.  However, as we detail in \S \ref{s:gradient}, the correct notion of a fixed point for generalized Ricci flow is that of a soliton, which will have vanishing Ricci curvature and constant, possibly nonzero, scalar curvature, and this is captured by a generalization of Perelman's $\gl$-functional.  For emphasis, we note that there are indeed even compact examples of generalized Ricci solitons with nonzero scalar curvature \cite{Streetssolitons, SU1}.
\end{rmk}

\begin{rmk}
  In the case of an exact Courant algebroid with trivial $H$-field, Proposition \ref{p:EHvariation} recovers Perelman's variational formula for the $\mathcal F$-functional \cite{Perelman1,OSW}.  In particular, here we set $\gs = e^{-\varphi} \gs_g$, and then $\gr_{\gm, \gs} = R + 4 \gD \varphi - 4 \brs{\N \varphi}^2 =: R^\varphi$.  Thus
\begin{align*}
    \eh (\gm, \gs) = \int_M \left( R + 4 \gD \varphi - 4 \brs{\N \varphi}^2 \right) e^{-2 \varphi} dV_g = \mathcal F(g, \varphi).
\end{align*}
Setting $\tdels{g} = h$ and $\tdels{\varphi} = \psi$ one has 
\begin{align*}
     \frac{d}{ds} \mathcal F( g_s,\varphi_s)  =&\ \int_M \left[ \IP{h, - \Rc - 2 \N^2 \varphi } + (\tfrac{1}{2} \tr_g h - 2 \psi) R^\varphi \right] e^{-2 \varphi} dV_g.
\end{align*}
It follows from further elementary calculations that
$\tdels{\gs}_g = \tfrac{1}{4} (\tr_g h) \gs_g$, so that
\begin{align*}
    \tdels{\gs} = \tfrac{1}{2} \left( \tfrac{1}{2} \tr_g h - 2 \psi \right) \gs.
\end{align*}
Thus we observe that the positive gradient flow for $\mathcal F$, noting the mixed sign inner product and weight on the half-density portion, is
\begin{align*}
    \dt g &= -2 (\Rc + 2 \N^2 \varphi), \qquad \dt \gs = - \tfrac{1}{2} R^\varphi \gs.
\end{align*}
\end{rmk}

\subsection{Generalized Bianchi identity}

Now using the invariance of the Einstein-Hilbert action under the Courant automorphism group we derive a general Bianchi identity for pairs $(\gm, \gs)$ as a consequence of this invariance.

\begin{prop}\label{prop:second_contracted_bianchi}
  Take $\gm\in\ms M$, $\sigma\in\ms H^*$, and $D\in LC(\gm,\div_\sigma)$. Then
    \[\gm\indices{^\alpha_\beta}D_\alpha\grc\indices{^\beta_\gamma}=\tfrac12D_\gamma\gr.\]
\end{prop}
\begin{proof}
    Fix a section $u\in\Gamma(E)$ and use it to generate a one-parameter family of Courant automorphisms $\Phi_s$.  Pulling back $\gm$ and $\gs$ by this family gives a one-parameter family $(\gm_s,\gs_s)$ satisfying
    \[\frac{\del}{\del s}\gm=\mc L_u\gm,\qquad \frac{\del}{\del s}\sigma=\mc L_u\sigma=L_{\rho(u)}\sigma.\]
    Since $\eh$ is invariant under the action of Courant automorphisms, using Proposition \ref{prop:killing} we have
    \begin{align*}
        0&=\frac{d}{ds} \eh(\gm_s, \gs_s) =\int_M \tfrac12(\mc L_u\gm)^{\alpha\beta}\grc_{\alpha\beta}\sigma^2+2\gr\sigma L_{\rho(u)}\sigma\\
        &=\int_M 2(D^{[\alpha}u^{\gamma]})\gm\indices{_\gamma^\beta}\grc_{\alpha\beta}\sigma^2+\gr L_{\rho(u)}\sigma^2=\int_M 2(D^{\alpha}u^{\gamma})\gm\indices{_\gamma^\beta}\grc_{\alpha\beta}\sigma^2-(L_{\rho(u)}\gr)\sigma^2.
    \end{align*}
    Using $\int (D_\beta v^\beta)\sigma^2=\int (\div_{\sigma} v)\sigma^2=0$ and $L_{\rho(u)}\gr=D_u\gr$ we then have
    \[0=\int_Mu^\gamma(-2\gm\indices{_\gamma^\beta}D^\alpha\grc_{\alpha\beta}-D_\gamma \gr)\sigma^2.\]
    Using $\gm\grc=-\grc\gm$ and the fact that the above holds for any $u$, the proposition follows.
\end{proof}

\section{Generalized Ricci flow}
\subsection{Definition as gradient flow of the Einstein--Hilbert functional}

\begin{defn}
        The \emph{generalized Ricci flow} is the gradient flow of $\tfrac12\on{grad}(\eh)$ on $\ms M\times \ms H^*$, where the gradient is taken using the metric from Definition \ref{d:L2metric}, i.e.
            \[\dt{\gm}=-2\grc_{\gm,\gs},\qquad \dt \sigma=-\tfrac{1}{2}\gr_{\gm,\gs} \sigma.\]
\end{defn}

\begin{rmk} We note that this equation is well-defined even if $\gm_t$ are generalized pseudometrics, and this point of view was explored in \cite{SeveraValach2}.  The fundamental analytic results to follow all require that $\gm$ is a generalized metric, which renders the generalized Ricci flow degenerate parabolic.
\end{rmk}

\begin{rmk} \label{r:grfwdivg}
    It is natural to consider an even more general equation defined using a general family of divergence operators.  In particular one can consider the generalized Ricci flow system
    \begin{align*} \label{f:GRFwdivg}
        \dt \gm &= -2 \grc_{\gm,\divg}, \qquad \dt \divg = - \differential \gr_{\gm,\divg}.
    \end{align*}
    All of the analytic theory we develop below applies for this system by adopting the ansatz (cf.\ Lemma \ref{lem:vary_divergence} below)
    \begin{align*}
        \dt \gm =&\ -2 \grc_{\gm,\divg}, \qquad \dt \gs = - \tfrac{1}{2} \gr_{\gm,\divg} \gs, \qquad \divg_t = \divg_0 + 2 \differential \log\frac{\gs}{\gs_0}.
    \end{align*}
\end{rmk}

\begin{rmk} \label{r:halfdensitygauge} For generalized metrics on transitive Courant algebroids there is always an induced classical Riemannian metric (Definition \ref{d:inducedmetricdilaton}), and thus an induced Riemannian half-density $\gs_g$, with associated divergence operator.  Through the use of Proposition \ref{p:Riccidivergencechange}, one sees that it is possible to gauge fix a solution to generalized Ricci flow as we have defined it to achieve a family of generalized metrics satisfying
\begin{align*}
    \dt \gm &= -2 \grc_{\gm, {\gs_g}}.
\end{align*}
There is an obvious appeal to this, as it seems we have removed one function from our system, and in the case of exact Courant algebroids with trivial $H$ field this is then the classic Ricci flow equation.  However, the analytic arguments to follow would be much more opaque from this point of view.  Furthermore, the flow of half-densities is strictly necessary to obtain the scalar curvature monotonicity to follow, and furthermore this flow very cleanly captures important relationships between the scalar curvature monotonicity, Perelman's Harnack estimates, and gauge invariance of the system.
\end{rmk}

\subsection{Induced flow of Riemannian metric and dilaton}
        Recall that for any $\gm\in\ms M^+$, the map $\hat\rho$ was defined as the inverse of $\rho|_{V_+}\colon V_+\to TM$ followed by the inclusion $V_+\to E$.
        \begin{lemma} \label{l:rhohatvariation}
            Given $\gm_s$ a family of generalized metrics, one has
            \[\dels{\hat\rho}=\tfrac12(1-\hat\rho\rho)\chi\hat \rho.\]
        \end{lemma}
        \begin{proof}
            First, note that for $s$ sufficiently small, $\gm = \gm_0$ and $\gm' = \gm_s$ satisfy that the map $u\mapsto \tfrac12(1+\gm'\!\gm)u$ is invertible and maps $V_+$ to $V_+'$. Using this we get that the invertible map
            \[\Phi\colon E\to E,\qquad \Phi(u)=\tfrac12(1+\hat\rho\rho\gm'\!\gm)(\tfrac12(1+\gm'\!\gm))^{-1}\]
            satisfies both $\rho\circ \Phi=\rho$ and $\Phi(V_+')=V_+$. We then get $\hat\rho'=\Phi^{-1}\circ \hat\rho$, and differentiating this in $s$ gives the result.
        \end{proof}
        
        \begin{prop}\label{prop:dilaton_transitive}
            Suppose $(\gm_t,\sigma_t)\in \ms M^+\times\ms H^*$ is a solution of the generalized Ricci flow.  Then the associated family of Riemannian metrics and dilatons $(g_t, \varphi_t)$ satisfy
            \begin{gather} \label{f:reducedflows}
            \begin{split}
            \dt g &= -2(\rho\grc\hat\rho)\cdot g, \\
            \dt\varphi&= \tfrac12\on{Tr}(\rho\grc\hat\rho)+\tfrac12\gr,
            \end{split}
            \end{gather}
            where $(A\cdot g)(X,Y):=-\tfrac12(g(A X,Y)+g(X,AY))$.
        \end{prop}
        \begin{proof}
            We directly calculate to get
            \begin{align*}
                \tdt g(X,X)&=\langle \hat\rho X,\tdt \hat\rho X\rangle=-\langle \hat\rho X,(1-\hat\rho\rho)\grc\hat\rho X\rangle=\langle \hat\rho X,\hat\rho \rho\grc\hat\rho X\rangle=2g(X,\rho\grc\hat\rho X).
            \end{align*}
            From $(A\cdot g)_{ij}=-A^k{}_{\vphantom{j}(i}g_{j)k}$ we have $g^{ij}(A\cdot g)_{ij}=-A^i{}_i$, and so
            \begin{align*}
                \tdt\varphi&=\sigma_g^{-1}\tdt\sigma_g-\sigma^{-1}\tdt\sigma=\tfrac14g^{ij}\tdt g_{ij}+\tfrac12\gr=\tfrac12\on{Tr}(\rho\grc\hat\rho)+\tfrac12\gr.\qedhere
            \end{align*}
        \end{proof}

\begin{ex} \label{ex:exact-flow}
  Let $E$ be an exact Courant algebroid. Choose an identification $E\cong TM\oplus T^*M$, corresponding to a particular (fixed) $H_0\in\Omega^3_{cl}(M)$. Then generalized metrics are given by graphs of $g+B$, with $g$ a Riemannian metric and $B\in\Omega^2(M)$. In terms of the variables $g,B,\varphi$, and defining $H:=H_0+dB$, the generalized Ricci flow takes the form
  \[\dt(g+B)=-2\rc^+-2 L_{\nabla\varphi}g-2i_{\nabla\varphi}H,\qquad \dt\varphi=\Delta_g \varphi-2(\nabla\varphi)^2+\tfrac1{12}H_{ijk}H^{ijk}.\]
  Decomposing the former, we obtain
  \[\dt g_{ij}=-2\rc_{ij}+\tfrac12H_{ikl}H_j{}^{kl}-4\nabla_i\nabla_j \varphi,\qquad \dt B_{ij}= \nabla^kH_{kij}-2(\nabla^k\varphi)H_{kij}.\]
  These are the usual generalized Ricci flow equations (cf.\ \cite{GRFbook}), together with a solution to the dilaton flow \cite{Streetsscalar}, which in terms of the variables here is $2 \varphi$.
\end{ex}

\begin{ex}
    A further example occurs in \cite{GF14}, where the explicit generalized Ricci flow equations are derived for a class of transitive Courant algebroids obtained by reduction (of equivariant exact Courant algebroids). The resulting equation can be expressed in terms of a metric $g$, a two-form $B$, a dilaton $\varphi$, and a principal $G$-connection $A$ with curvature $F$ as
\begin{align*}
    \dt g_{ij}&=-2\rc_{ij}-4\nabla_i\nabla_j\varphi+\tfrac12H\indices{_i^k^l}H_{jkl}+\tfrac12\on{Tr} F_{ik}F\indices{_j^k},\\
    \dt B_{ij}&=\nabla^kH_{kij}-2(\nabla^k\varphi)H_{kij},\qquad H=H_0+dB\\
    \dt\varphi&=\Delta_g\varphi-2(\nabla\varphi)^2+\tfrac1{12}H_{ijk}H^{ijk}+\tfrac{1}{16}\on{Tr}F_{ij}F^{ij},\\
    \dt A_i&=\nabla^kF_{ki}-2(\nabla^k\varphi)F_{ki}-\tfrac12H_{ijk}F^{jk},
\end{align*}
where $\nabla$ is the combination of the Levi-Civita connection with $A$, and $\on{Tr}(XY)$ is an invariant inner product on $\mf g$. Thus it is a further coupling of the generalized Ricci flow on exact Courant algebroids (\ref{f:exactGRF}) to Yang-Mills flow.
\end{ex}

\subsection{Scalar curvature monotonicity}

\begin{lemma}\label{lem:vary_divergence}
            Given $\sigma_s\in\ms H^*$ such that $\frac{\del}{\del s} \gs = \nu$, one has
            \begin{align*}
\dels \div_{\sigma} = 2\differential(\tfrac{\nu}{\sigma}).
            \end{align*}
        \end{lemma}
        \begin{proof}
            Since $\div_{e^f\sigma}-\div_{\sigma}=2\differential f$, we compute
            \begin{align*}
            \tdels \div_{\sigma} = \tdels (\div_{\sigma}-\div_{\sigma_0})=2 \tdels (\differential \log\tfrac{\sigma}{\sigma_0}) =2\differential(\tdels \log\tfrac{\sigma}{\sigma_0})=2\differential(\tfrac{\nu}{\sigma}),
            \end{align*}
            as claimed.
        \end{proof}

\begin{thm} \label{t:scalarcurvmon} (cf.\ Theorem \ref{t:scalarevolutionintro})
            Let $(\gm_t, \gs_t)$ be a solution to generalized Ricci flow on a compact manifold.  Then for all smooth existence times $t$ one has
            \[\dt\gr= \gD_{\gm,\gs} \gr+\,\vert\!\grc\!\vert_{\gm}^2.\]
\end{thm}

\begin{proof} We note that by Lemma \ref{lem:vary_divergence}, one has $\dt \divg_{\gs} = - \differential \gr$.  Thus substituting $\chi=-2\grc$ and $\varepsilon=-\differential\gr$ into Proposition \ref{prop:variation} we obtain
            \[\dt\gr=-2\gm\indices{^\alpha_\gamma} \gm\indices{^\beta_\delta} D_\alpha D_\beta\grc^{\gamma\delta}+2\gm^{\alpha\beta}D_\alpha D_\beta\gr-\grc^{\alpha\beta}\grc_{\alpha\beta}.\]
            The result then follows from Proposition \ref{prop:second_contracted_bianchi} and the fact $\grc^{\alpha\beta}\grc_{\alpha\beta}=-\,\vert\!\grc\!\vert_{\gm}^2$.
\end{proof}

\begin{cor} \label{c:scalarlb} Let $(\gm_t, \gs_t)$ be a solution to generalized Ricci flow on a compact manifold with strictly positive initial data.  Then for all smooth existence times $t > 0$ one has
\begin{align*}
    \inf_{M \times \{t\}} \gr \geq \inf_{M \times \{0\}} \gr.
\end{align*}
\begin{proof} Note that the operator $\gD_{\gm,\gs}$ is strictly elliptic.  Furthermore as the metrics are strictly positive, one has $\brs{\grc}_{\gm}^2 \geq 0$, thus $\gr$ is a supersolution of a strictly parabolic equation.  The result follows from the maximum principle.
\end{proof}
\end{cor}

\section{Short-time existence and uniqueness}

In this section we turn to establishing analytic properties of the generalized Ricci flow.  For generalized metrics as in Definition \ref{d:genmetric}, we will see that the generalized Ricci flow is a degenerate parabolic equation, with degeneracy entirely arising from the invariance of the equation under Courant automorphisms.  Through an application of DeTurck's gauge fixing method we will establish the existence.  We emphasize that the results of this section apply to generalized metrics, and do not require strict positivity.

\subsection{Generalized Bianchi operator}

Recall from Proposition \ref{prop:variation} that
the linearization of the operator \[\mathcal \grc \colon \ms M^+ \times \ms H^* \to \End(E)\] has principal symbol
\begin{align*}
    P(L \grc)(\chi, \ge) = - \tfrac{1}{2}\gD \chi - \tfrac12 \mc L_{\on{Tr} D \chi+\gm \varepsilon}\gm.
\end{align*}
Thus the Ricci operator is elliptic modulo the image of the Courant automorphism group.  To further prepare for the proof of short-time existence we define an explicit differential operator which deals with the piece of the linearization tangent to the gauge orbit.

\begin{defn} Given $\gm, \divg$, fix $D \in LC(\gm, \divg)$ and $\til{D}$ an arbitrary Courant algebroid connection.  For this data define the \emph{Bianchi operator} $\BB\in \Gamma(E^*)\cong \Gamma(E)$ by
\begin{align*}
\BB_\gamma:=\gm^{\alpha\beta}A_{\ga \gg \gb},
\end{align*}
where $A = D - \til{D}$, noting (\ref{f:connectionconvention}) for the index ordering of $A$.
\end{defn}

\begin{lemma} \label{l:Wdef}
    For any $u\in\Gamma(E)$ we have
    \[\langle \BB,u\rangle= - \div\gm \!u+\on{Tr}(\gm \tilde Du),\]
    where $\on{Tr}(\gm \tilde Du) = \gm^{\ga \gb} \til{D}_{\ga} u_{\gb}$.
In particular, $\BB$ does not depend on the representative $D\in LC(\gm,\div)$.
\end{lemma}
\begin{proof} We compute
    $\BB_\gamma u^\gamma= - \gm^{\alpha\beta}D_\alpha u_\beta+\gm^{\alpha\beta}\tilde D_\alpha u_\beta= - D_\alpha (\gm^{\alpha\beta}u_\beta)+\gm^{\alpha\beta}\tilde D_\alpha u_\beta$.
\end{proof}
Due to Lemma \ref{l:Wdef} we will thus write
\[\BB(\gm,\div,\tilde D).\]  The key point is that the principal symbol of $\BB$ is precisely the missing piece of the symbol of $\grc$.  In the statement of the lemma below we allow for a general variation $\varepsilon
$, of the divergence operator.   Although the result is zeroth-order in $\varepsilon$, we will treat this as a first order operator as in context we will always have $\varepsilon = \differential \psi$ for some smooth function $\psi$ (cf.\ Lemma \ref{lem:vary_divergence}).

\begin{nota}
    To simplify notation, we will employ the symbol $\equiv$ to denote equality up to terms of lower order in derivatives of $\chi$ and $\varepsilon$ (treating $\varepsilon$ as already containing one derivative, in accordance with the above remark).
\end{nota}

\begin{lemma} \label{l:Wlinearization} 
Given $(\gm, \divg)$, $D\in LC(\gm,\div)$, and a fixed background connection $\tilde D$, we have
\begin{align*}
    P [L {\BB}(\gm,\divg,\tilde D)](\chi, \varepsilon) = - \on{Tr} D \chi - \gm \varepsilon.
\end{align*}
\begin{proof} 
    From formulas \eqref{eq:variation_trace} we compute
    \begin{align*}
        \frac{\del}{\del s} \BB_c &\equiv \gm^{\alpha\beta}A_{\alpha c \beta }=A\indices{^a_{ca}}-A\indices{^{\hat a}_{c \hat a}}=- (\tfrac12D^{\hat a}\chi_{c\hat a}+\varepsilon_c)-\tfrac12D^{\hat a}\chi_{c\hat a}=-D^{\hat a}\chi_{c\hat a}-\varepsilon_c,\\
        \frac{\del}{\del s} \BB_{\hat c} &\equiv \gm^{\alpha\beta}A_{\alpha \hat c \beta}=A\indices{^a_{\hat c a}}-A\indices{^{\hat a}_{\hat c \hat a}}=- \tfrac12D^{a}\chi_{\hat c a}-(\tfrac12D^{a}\chi_{\hat c a}-\varepsilon_{\hat c})=-D^{a}\chi_{\hat c a}+\varepsilon_{\hat c}.\qedhere
    \end{align*}
\end{proof}
\end{lemma}

\begin{ex}
    Let $E$ be an exact Courant algebroid with a family $\gm_s$ of generalized metrics and a family of dilatons $\varphi_s$. Fix the (unique) identification $E\cong TM\oplus T^*M$ for which $\gm_0$ is the graph of a symmetric form. Pick a frame $E_i$ of $TM$. Denoting the Levi-Civita connection for $g_0$ by $\nabla$, we then have at $s=0$:
    \[\dels \BB_k\equiv - \nabla^i\tdels B_{ik},\qquad \dels \BB^k\equiv \nabla^ig^{jk}\tdels g_{ij}-\tfrac12\nabla^k\on{tr}_g\tdels g +2\nabla^k\tdels\varphi.\]
\end{ex}

\subsection{Dilaton operator}

Our existence theory for generalized Ricci flow works using the equivalent data of a generalized metric and the associated dilaton.  In Proposition \ref{prop:dilaton_transitive} we derived the evolution equation for the dilaton, a delicate operator involving the inverse of the anchor map, and the Ricci and scalar curvatures.  Here we derive the crucial delicate point that the symbol of the linearization of this operator in the generalized metric and dilaton is simply the Laplace operator acting on the linearization of the dilaton.  In particular, there are no second order terms in the linearization of the generalized metric, which renders the gauge-fixed generalized Ricci flow upper-triangular and hence elliptic.

\begin{defn} \label{d:dilatonoperator}
  Define the \emph{dilaton operator} by
  \begin{align*}
    \dil&\colon \ms M^+ \times \ms H^* \to C^{\infty}(M)\\
    \dil_{\gm, \gs} &= \tfrac12\on{Tr}(\rho\grc_{\gm,{\gs}}\hat\rho)+\tfrac12\gr_{\gm,{\gs}}.
  \end{align*}
\end{defn}

\begin{lemma} \label{l:dillinearization} The linearization of the operator $\dil$ has principal symbol
\begin{align*}
    P[L \dil_{\gm,\gs}](\chi, \psi) = \gD \psi.
\end{align*}
\end{lemma}

\begin{proof} We fix a one-parameter family $(\gm_s, \gs_s)$ of metrics and half-densities with an associated family of dilatons $\varphi_s$.
  Using Lemma \ref{lem:vary_divergence} and the proof of Proposition \ref{prop:dilaton_transitive}, the corresponding variation $\varepsilon$ of the divergence is, up to lower order terms in $\chi$ and $\psi$,
  \[\varepsilon_\alpha=2D_\alpha\!\left(\sigma^{-1} \tdels \sigma\right)=-2D_\alpha \psi-\tfrac12D_\alpha(\on{Tr}\rho\chi\hat\rho)\equiv -2\rho(e_\alpha)\psi-\tfrac12\rho^{i\hat a}\hat\rho\indices{_i^a}\rho(e_\alpha)\chi_{a\hat a}.\]
  Since $\rho\circ \rho^*=0$, $\rho\indices{^i_c}\hat\rho\indices{_i^a}=\delta^a_c$, and $\hat\rho\indices{_i^a}\rho\indices{^j_{a}}=\delta^j_i$, we have
  \[\rho^{i\hat a}\hat\rho\indices{_i^a}\rho\indices{^j_{\hat a}}=\rho^{i\alpha}\hat\rho\indices{_i^a}\rho\indices{^j_{\alpha}}-\rho^{ic}\hat\rho\indices{_i^a}\rho\indices{^j_{c}}=-\rho^{ja}, \qquad \rho^{i\hat a}\hat\rho\indices{_i^a}\rho\indices{^j_{a}}=\rho^{j\hat a}.\]
  We can then directly calculate, up to lower order terms in $\chi$ and $\psi$,
  \begin{align*}
    \tdels \dil \equiv&\ \tfrac12\rho^{i\hat a}\hat \rho\indices{_i^a}\tdels \grc_{\hat aa}+\tfrac12 \tdels\gr\\
    \equiv&\ \tfrac12\rho^{i\hat a}\hat \rho\indices{_i^a}(-\tfrac12\gm^{\alpha\beta}\rho(e_\alpha)\rho(e_\beta)\chi_{a\hat a}+\rho(e_c)\rho(e_a)\chi\indices{^c_{\hat a}}-\rho(e_{\hat c})\rho(e_{\hat a})\chi\indices{_a^{\hat c}}\\
    &\quad-\rho(e_a)\varepsilon_{\hat a}-\rho(e_{\hat a})\varepsilon_a)+\tfrac12(-\rho(e_\alpha)\rho(e_\beta)\chi^{\alpha\beta}-2\gm^{\alpha\beta}\rho(e_\alpha)\varepsilon_\beta)\\
    \equiv&\ \tfrac12\rho^{i\hat a}\hat \rho\indices{_i^a}\rho(e_c)\rho(e_a)\chi\indices{^c_{\hat a}}-\tfrac12\rho^{i\hat a}\hat \rho\indices{_i^a}\rho(e_{\hat c})\rho(e_{\hat a})\chi\indices{_a^{\hat c}}+\rho^{i\hat a}\hat \rho\indices{_i^a}\rho(e_a)[\rho(e_{\hat a})\psi+\tfrac14\rho^{i\hat c}\hat\rho\indices{_i^c}\rho(e_{\hat a})\chi_{c\hat c}]\\
    &\quad+\rho^{i\hat a}\hat \rho\indices{_i^a}\rho(e_{\hat a})[\rho(e_{a})\psi+\tfrac14\rho^{i\hat c}\hat\rho\indices{_i^c}\rho(e_{a})\chi_{c\hat c}]-\tfrac12\rho(e_\alpha)\rho(e_\beta)\chi^{\alpha\beta}+2\gm^{\alpha\beta}\rho(e_\alpha)\rho(e_\beta)\psi\\
    &\quad+\tfrac14\rho^{i\hat a}\hat \rho\indices{_i^a}\gm^{\alpha\beta}\rho(e_\alpha)\rho(e_\beta)\chi_{a\hat a}\\
    \equiv&\ \tfrac12\rho(e_c)\rho(e^{\hat a})\chi\indices{^c_{\hat a}}+\tfrac12\rho(e_{\hat c})\rho(e^a)\chi\indices{_a^{\hat c}}-\rho(e_a)[\rho(e^a)\psi+\tfrac14\rho^{i\hat c}\hat\rho\indices{_i^c}\rho(e^a)\chi_{c\hat c}]\\
    &\quad+\rho(e_{\hat a})[\rho(e^{\hat a})\psi+\tfrac14\rho^{i\hat c}\hat\rho\indices{_i^c}\rho(e^{\hat a})\chi_{c\hat c}]-\tfrac12\rho(e_\alpha)\rho(e_\beta)\chi^{\alpha\beta}+2\gm^{\alpha\beta}\rho(e_\alpha)\rho(e_\beta)\psi\\
    &\quad+\tfrac14\rho^{i\hat a}\hat \rho\indices{_i^a}\gm^{\alpha\beta}\rho(e_\alpha)\rho(e_\beta)\chi_{a\hat a}\\
    \equiv&\ \gm^{\alpha\beta}\rho(e_\alpha)\rho(e_\beta)\psi \equiv \gD \psi,
  \end{align*}
  as claimed.
\end{proof}

\subsection{Existence}

We are now ready to give the proof of short-time existence of solutions to generalized Ricci flow.  The main idea is conceptually identical to the well-known DeTurck gauge-fixing argument \cite{DeTurck1}.  In particular, we construct a differential operator given by the generalized Ricci operator plus the generalized Lie derivative with respect to the Bianchi operator $\BB$.  As explained in the previous subsection, there is a further subtlety where we are forced to make a change of variables and understand the equation through the equivalent data of the generalized metric and the dilaton. The gauge-modified operator on this data is strictly elliptic, thus by standard arguments we produce a short-time solution to the gauge-modified flow.  Pulling everything back by the relevant family of Courant automorphisms we obtain the required solution to generalized Ricci flow.

\begin{prop} \label{p:ste} Given $M$ a compact manifold and $E \to M$ a Courant algebroid, fix a generalized metric $\gm$ and a half-density $\gs\in\ms H^*$.  Then there exists $\ge > 0$ and a solution to generalized Ricci flow with initial data $(\gm, \gs)$ on $[0, \ge)$.
\begin{proof} Fix an auxiliary background connection $\til{D}$.   Using this we define an operator
\begin{align*}
    {F} &\colon \ms M^+\times C^{\infty}(M) \to \on{End}(E)\times C^{\infty}(M)\\
    {F} &(\gm, \varphi) := \left( -2 \grc_{\gm, \gs} + \LL_{\mathcal B(\gm, \gs, \til{D})} \gm, \;\mathcal{Dil}_{\gm,\gs} + \LL_{\mathcal B(\gm, \gs, \til{D})} \varphi \right),
\end{align*}
where as usual $\gs = e^{-\varphi} \gs_g$.  By the results of Proposition \ref{prop:variation}, and Lemmas \ref{l:dillinearization} and \ref{l:Wlinearization}, it follows directly that the principal symbol of the linearization of ${F}$ is
\begin{align*}
    P[ L {F}(\gm, \varphi)](\chi, \psi) = \left(\begin{matrix}
        \Delta \chi & \star \\
        0 & \Delta \psi
    \end{matrix} \right).
\end{align*}
Thus the operator $F$ is strictly elliptic by Proposition \ref{p:Laplacianelliptic}.  We pause to emphasize that the linearization of the generalized Ricci tensor does have a generic second order term involving the linearization of the dilaton.  Thus the result of Lemma \ref{l:dillinearization}, showing that there is no second order term in the linearization of the metric, is crucial to obtaining that this system is elliptic with upper-triangular symbol as shown.  It follows from the standard theory of parabolic equations on compact manifolds that for given $\gm, \varphi$, there exists $\ge > 0$ and a solution to
\begin{align} \label{f:deturckflow}
    \dt (\gm, \varphi) = F(\gm, \varphi)
\end{align}
on $[0,\ge)$, yielding an associated solution to gauge-fixed generalized Ricci flow $(\gm_t, \gs_t)$.  Now it follows from standard ODE results that we can solve for the one-parameter family of Courant automorphisms (cf.\ Proposition \ref{prop:sections_automorphisms})
\begin{align} \label{f:DeTurckODE}
    \dt \Phi &= - Z_{\mathcal B(\gm, \gs, \til{D})}, \qquad \Phi_0 = \Id,
\end{align}
and by a standard computation using the naturality of Ricci and scalar curvatures under Courant automorphisms it follows that $(\Phi_t^* \gm_t, \Phi_t^* \gs_t)$ is the required solution to generalized Ricci flow.
\end{proof}
\end{prop}

\subsection{Uniqueness}

    In \cite{HamiltonSing} Hamilton gave a conceptual proof of the uniqueness of smooth solutions to Ricci flow on compact manifolds.  The idea is to recover the family of diffeomorphisms used to relate the Ricci--DeTurck flow to Ricci flow in the proof of existence a priori from a solution to Ricci flow.  These DeTurck diffeomorphisms, constructed as the solution of an ODE in terms of the solution of Ricci--DeTurck flow, become a solution to a PDE when expressed with respect to pullback data, i.e.\ the associated solution of Ricci flow.  Surprisingly, this PDE is precisely the harmonic map heat flow, and thus one can solve this parabolic flow on a Ricci flow background to construct the needed diffeomorphisms.

    We show the uniqueness of solutions to generalized Ricci flow by extending this circle of ideas to the Courant automorphism group.  The basic idea is the same: to express the family of Courant automorphisms relating the gauge-fixed and standard generalized Ricci flows, constructed as an ODE relative to the gauge-fixed data, as the solution of a certain parabolic PDE relative to the solution of generalized Ricci flow.  As we will see below in Lemma \ref{l:Hlinearization}, the relevant operator at the level of sections of $E$ is degenerate elliptic, evidenced by the presence of the term $\differential \divg u$ in the linearization.  Such a degeneracy is to be expected, due to the fact that the map $u \mapsto Z_u$ taking sections to infinitesimal automorphisms has an infinite-dimensional kernel containing exact sections $\differential f$, i.e.\ $Z_{\differential f} = 0$. (cf.\ Proposition \ref{prop:sections_automorphisms}).  However, the operator is elliptic in terms of the induced infinitesimal automorphisms $Z_u$.

    \begin{lemma} \label{l:Hlinearization} Given $\gm, \gs$ a choice of generalized metric and half-density, fix $D \in LC(\gm, \divg_{\gs})$ and $\til{D}$ a background Courant algebroid connection.  Define 
    \begin{align*}
        \mathcal H \colon& \Aut(E) \to \gG(E), \qquad \mathcal H(\Phi) = \BB(\Phi_* \gm, \Phi_* \gs, \til{D}).
    \end{align*}
    Then
    \begin{align*}
        P[L \mathcal H(\Id)](Z_u) &= - \gD u + \differential \divg_{\gs}(\gm u).
    \end{align*}
    \end{lemma}
    
    \begin{proof} As the Bianchi operator is equivalently expressed in terms of a choice of Levi-Civita connection we use the given choice $D$ and then observe that up to lower order derivatives in $u$ one has
            \[\langle e_\gamma,(\mc L_u D)_{e_\alpha}e_\beta\rangle\equiv -\langle e_\gamma,D_{e_\alpha}[u,e_\beta]\rangle\equiv -\langle e_\gamma,D_{e_\alpha}(\rho^*du_\beta-[e_\beta,u])\rangle\equiv \rho(e_\alpha)\rho(e_\beta)u_\gamma- (\gamma \leftrightarrow\beta).\]
Using this and the definitions of $\mathcal H$ and $\mathcal B$ we have
        \begin{align*}
            P[L \mathcal H(\Id)](Z_u)_\gamma &\equiv -\BB(\gm,\mc L_uD,\til{D})_\gamma\\
            &\equiv -\gm\indices{^\alpha^\beta}(\mc L_u D)\indices{_\alpha_\gamma_\beta}\\
            &\equiv \gm^{\alpha\beta}(\rho(e_\alpha)\rho(e_\gamma)u_\beta-\rho(e_\alpha)\rho(e_\beta)u_\gamma)\\
            &\equiv - (\Delta u)_\gamma+ (\differential \div_{\gs}  \gm u)_\gamma,
        \end{align*}
    as claimed.
    \end{proof}

\begin{prop} \label{p:uniqueness} Given $M$ a compact manifold and $E \to M$ a Courant algebroid, fix a generalized metric $\gm$ and a half-density $\gs\in\ms H^*$.  The solution to generalized Ricci flow constructed in Proposition \ref{p:ste} is unique.
    \begin{proof}
        Let $(\gm_t, \gs_t)$ be any solution to generalized Ricci flow with the given initial data.  We fix a background connection $\til{D}$, and our goal is to reconstruct the family of Courant automorphisms used in Proposition \ref{p:ste}, but only knowing a priori a given solution to generalized Ricci flow.  In particular, we claim that there exists $\ge > 0$ and a unique solution to
        \begin{align} \label{f:CAflow}
            \dt \Phi &= - Z_{\mathcal H(\Phi)}, \qquad \Phi_0 = \Id.
        \end{align}
        on $[0,\ge)$.  Assuming this claim for the moment, it follows by construction that the data $((\Phi_t)_* \gm_t, (\Phi_t)_* \gs_t)$ is a solution to (\ref{f:deturckflow}).  A standard argument using the naturality of the Ricci and scalar curvatures, and the uniqueness of solutions to (\ref{f:deturckflow}) and (\ref{f:CAflow}) finishes the proof.

        To prove the claim, it suffices to show that equation (\ref{f:CAflow}) is parabolic at time zero.  Recall that the objects $Z_u$ are vector fields on the total space of $E$, thus to apply the standard theory of parabolic equations, we must reformulate equation (\ref{f:CAflow}) as an equation of sections of vector bundles over $M$.  To address this we first recall some standard constructions.  Take $P$ to be the bundle of orthonormal frames for the signature $(p,q)$ inner product $\IP{\cdot,\cdot}$ on $E$; this is a principal $G:=O(p,q)$-bundle.  We then obtain the associated Atiyah algebroid, i.e.\ the Lie algebroid \[\at(E):=TP/G.\]
        The key point is that sections of $\at(E)$ are canonically identified with $G$-invariant vector fields on $P$. The bracket on $\at(E)$ corresponds to the commutator of these vector fields, while the anchor map is given by the pushforward along the projection $P\to M$. Using $P\times_G\mf g\cong \w2E$, the Atiyah algebroid fits into the exact sequence of vector bundles
    \[\begin{tikzcd}
        0 \ar[r]& \w2E\ar[r]&\at(E)\ar[r]& TM\ar[r]&0.
    \end{tikzcd}\]
    A Courant algebroid connection $D$ can be equivalently described by a vector bundle map \[\lambda^D\colon E\to \at(E)\] which preserves the anchor, i.e.\ makes
    \[\begin{tikzcd}
        &&E\ar[d]\ar[dr]&&\\
        0 \ar[r]& \w2E\ar[r]&\at(E)\ar[r]& TM\ar[r]&0
    \end{tikzcd}\]
    commutative. Similarly, for any section $u\in\Gamma(E)$, the infinitesimal automorphism $Z_u$ induces a $G$-invariant vector field on $P$ and hence a section of $\at(E)$. Let us call this map \[Y\colon \Gamma(E)\to \Gamma(\at(E)).\]
    Explicitly, if $D$ is any torsion-free connection, equation \eqref{eq:torsion} implies
    \[Y(u)=\lambda^D(u)+2\mu(Du),\]
    where the vector bundle map $\mu\colon E\otimes E\to \at(E)$ is the composition of the antisymmetrisation $E\otimes E\to \w2E$ with the canonical inclusion $\w2E\to \at(E)$.
    
    Let now $D$ be a Courant algebroid connection and $D^{\at}$ a Courant algebroid connection on the vector bundle $\at(E)$, i.e.\ a map \[D^{\at}\colon \Gamma(E)\times\Gamma(\at(E))\to\Gamma(\at(E))\] satisfying $D^{\at}_{fu}a=fD^{\at}_{u}a$ and $D^{\at}_{u}(fa)=fD^{\at}_{u}a+(\rho(u)f)a$. Using the Leibniz rule, we extend the action of $D$ and $D^{\at}$ to an operator (which we will also call $D$) acting on any $\Gamma(E^{\otimes k}\otimes \at(E)^{\otimes l})$.  This then yields the induced operator  $\Delta:=\gm^{\alpha\beta}D_\alpha D_\beta$.  
    
    With this background at hand, we compute using Lemma \ref{l:Hlinearization} and Proposition \ref{prop:sections_automorphisms},
    \begin{align*}
        P[ L( - Y(\mathcal H(\Id)))] (Y(u)) &\equiv Y(\Delta u - \differential \divg_{\gs} (\gm u)) \equiv Y(\Delta u) \equiv \Delta Y(u).
    \end{align*}
    The claim of unique short-time existence for solutions to (\ref{f:CAflow}) follows, finishing the proof.
    \end{proof}
\end{prop}

\begin{proof}[Proof of Theorem \ref{t:mainthm1}] The result follows from Propositions \ref{p:ste} and \ref{p:uniqueness}.
\end{proof}

\section{Generalized Ricci flow as a gradient flow} \label{s:gradient}

In this section we give the gradient flow interpretation of generalized Ricci flow, adopting the approach of Perelman \cite{Perelman1}.  For emphasis, while the formal computations hold for generalized Ricci flows of arbitrary pseudometrics, the claims of monotonicity all require the hypothesis of strict positivity.

\subsection{Steady Harnack Estimate}

Along $(\gm_t, \gs_t)$ a solution to generalized Ricci flow we define the time-dependent heat operator
\begin{align*}
    \square := \dt - \gD_{\gm,\gs}.
\end{align*}
Furthermore we define the conjugate heat operator relative to $\gs^2$:
\begin{align*}
    \square^*_{\gs} := - \dt - \gD_{\gm, \gs} + \gr_{\gm,\gs}.
\end{align*}

\begin{lemma} \label{l:conjheat} Given $(\gm_t, \gs_t)$ a solution to generalized Ricci flow, for smooth functions $\phi, \psi$ one has
\begin{enumerate}
    \item $\frac{d}{dt} \int_M \phi \psi \gs^2 = \int_M \left( \psi \square \phi - \phi \square^*_{\gs} \psi \right) \gs^2$.
    \item A solution to $\square^*_{\gs} \phi = 0$ preserves mass against $\gs^2$, i.e.\ $\frac{d}{dt} \int_M \phi \gs^2 = 0$.
\end{enumerate}
\begin{proof} Given $\phi, \psi$, we first observe that as a consequence of $\int_M (\divg_{\gs} e) \gs^2 = 0$ one has that the Laplacian $\gD_{\gm,\gs}$ is self-adjoint with respect to $\gs^2$, i.e.
\begin{align*}
    0 =&\ \int_M \left( \psi \gD_{\gm, \gs} \phi - \phi \gD_{\gm,\gs} \psi \right) \gs^2.
\end{align*}
Using this we compute
\begin{align*}
    \frac{d}{dt} \int_M \phi \psi \gs^2 =&\ \int_M \left( \psi \dt \phi + \phi \dt \psi - \phi \psi \gr_{\gm,\gs} \right) \gs^2\\
    =&\ \int_M \left[ \psi \left( \dt - \gD_{\gm,\gs} \right) \phi - \phi \left( - \dt - \gD_{\gm,\gs} + \gr_{\gm,\gs} \right) \psi \right] \gs^2,
\end{align*}
as required for item (1).  Item (2) is then an elementary consequence.
\end{proof}
\end{lemma}

\begin{prop} \label{p:steadyharnack} Given $(\gm_t, \gs_t)$ a solution to generalized Ricci flow, let $u = e^{-f}$ be a solution to $\square^*_{\gs} u = 0$.  Furthermore, let $\gs^f := e^{-f/2} \gs$.  Then
\begin{align*}
    \square^*_{\gs} \left( \gr(\gm, \gs^f) u \right) &= - \brs{\grc(\gm, \gs^f)}^2_{\gm} u.
\end{align*}
\begin{proof} Given the data of the proposition, apply the time-dependent gauge transformation generated by $\gm \differential f$, denoting these quantities still $(\gm, \gs, f)$.  We claim that the new one-parameter family $(\gm_t, \gs^f_t)$ satisfies
\begin{align*}
    \dt \gm = - 2 \grc(\gm, \gs^f), \qquad \dt \gs^f =&\ 0.
\end{align*}
The first equation follows from Proposition \ref{p:Riccidivergencechange}.  For the second we first observe the evolution equation for the gauge-fixed half-density and Lemma \ref{l:weightedlaplacian}
\begin{align*}
    \dt \gs =&\ - \tfrac{1}{2} \gr_{\gm,\gs} \gs - \mc L_{\gm \differential f} \gs = \left[ - \tfrac{1}{2} \gr_{\gm,\gs} - \tfrac{1}{2} \gD_{\gm,\gs} f \right] \gs.
\end{align*}
Also, using the gauge-fixed conjugate heat equation we compute 
\begin{align*}
\dt f &=- \gD_{\gm, \gs} f + \brs{\differential f}^2_{\gm} - \gr_{\gm,\gs} - \brs{\differential f}^2_{\gm} = - \gD_{\gm,\gs} f - \gr_{\gm,\gs}.
\end{align*}
The claim $\dt \gs^f = 0$ follows immediately.
Given this, we have from Proposition \ref{prop:variation} the evolution equation
\begin{align*}
    \dt \gr(\gm, \gs^f) &= - \gD_{\gm,\gs^f} \gr(\gm,\gs^f) + \brs{\grc(\gm, \gs^f)}_{\gm}^2\\
    &= - \gD_{\gm, \gs} \gr(\gm,\gs^f) + \IP{\differential \gr(\gm,\gs^f), \gm \differential f} + \brs{\grc(\gm, \gs^f)}_{\gm}^2.
\end{align*}
As this is an equation of scalars, for the original flow (before gauge-fixing), we conclude
\begin{align*}
    \dt \gr(\gm, \gs^f) = - \gD_{\gm,\gs} \gr(\gm,\gs^f) + 2 \IP{\differential \gr(\gm,\gs^f),\gm \differential f} + \brs{\grc(\gm,\gs^f)}^2_{\gm}.
\end{align*}
We can now compute
\begin{align*}
    \square^*_{\gs}(\gr(\gm,\gs^f) u) &= \left(- \dt - \gD_{\gm,\gs} + \gr(\gm,\gs) \right) \left[ \gr(\gm, \gs^f) u \right]\\
    &= \left[ \left( - \dt - \gD_{\gm, \gs} \right) \gr(\gm,\gs^f) \right] u + \gr(\gm,\gs^f) \square^*_{\gs} u  - 2 \IP{\differential \gr(\gm,\gs^f),\gm \differential u}\\
    &= -\brs{\grc(\gm, \gs^f)}^2_{\gm} u,
\end{align*}
as claimed.
\end{proof}
\end{prop}

\subsection{\texorpdfstring{$\gl$}{lambda}-functional monotonicity}

\begin{defn} Given $E$ a Courant algebroid over a compact base $M$, define
\begin{align*}
    \gl& \colon \ms M^+ \to \mathbb R, \qquad 
    \gl(\gm) := \inf_{ \{\gs \in \ms H^* \mid\, \int \gs^2 = 1 \}} \eh(\gm, \gs).
\end{align*}
\end{defn}

\begin{rmk} It follows from Proposition \ref{p:scalardivchange} and a standard argument (cf.\ \cite{KL} Proposition 7.1) that for any metric $\gm$ there exists a unique $\gs$ achieving the infimum defining $\gl$.  In particular, by Remark \ref{r:EHLaplace} such $\gs$ satisfies $\check{\gD}_{\gm} \gs = 0$, a self-adjoint elliptic equation.
\end{rmk}

\begin{defn} \label{d:soliton} A generalized metric $\gm$ is a \emph{soliton} if for the unique $\gs$ achieving the infimum defining $\gl (\gm)$, one has $\grc(\gm,\gs) =0$.
\end{defn}

\begin{rmk} We note that by the Bianchi identity of Proposition \ref{prop:second_contracted_bianchi}, a soliton $\gm$ automatically has $\gr(\gm, \gs)$ constant.  Furthermore, by Proposition \ref{p:Riccidivergencechange}, it follows that for arbitrary initial half-density $\gs_0$, the generalized Ricci flow with initial data $(\gm, \gs_0)$ will evolve $\gm$ purely by Courant automorphism pullback.  Also, by adapting \cite{Streetsscalar} Proposition 3.8 one can show that the half-density will converge after normalization and gauge-fixing to $\gs$.  
\end{rmk}

\begin{lemma}
    Fix $\gm\in\ms M^+$. If there exists a volume-normalized half-density $\gs$ such that $\grc(\gm, \gs) = 0$, then this $\gs$ achieves the infimum in the definition of $\gl$. Thus $\gm$ is a soliton.
\end{lemma}
\begin{proof}
    If $\grc(\gm,\gs)=0$, Proposition \ref{p:EHvariation} and the constancy of $\gr(\gm,\gs)$ imply that the variation of the Einstein--Hilbert functional at $(\gm,\gs)$ is
    \[\dels \eh (\gm,\gs) =-(\nu,\sigma)_{\gm,\sigma}\gr(\gm,\gs).\]
    Since this vanishes when the variation is restricted to normalised half-densities, $\lambda$ attains an extremum at $\sigma$. The claim then follows from the uniqueness of the extrema.
\end{proof}

\begin{thm} \label{t:gradient} (cf.\ Theorem \ref{t:gradientintro}) For strictly positive metrics, generalized Ricci flow is the gradient flow of $\gl$.  More precisely, given $(\gm_t, \gs_t)$ a solution to generalized Ricci flow with strictly positive initial data over a compact manifold $M$, for any smooth existence times $t_1 < t_2$ one has $\gl(\gm_{t_1}) \leq \gl(\gm_{t_2})$.  Furthermore, equality holds if and only if $\gm_{t}$ is a soliton.
\end{thm}
\begin{proof} Fix existence times $t_1 < t_2$, and at time $t_2$ let $f_{t_2}$ be the unique function such that $\gs_{t_2}^{f_{t_2}}$ achieves the infimum in the definition of $\gl(\gm_{t_2})$.  As the conjugate heat operator is a linear parabolic equation in backwards time, there exists a unique solution $u = e^{-f}$ to $\square^*_{\gs} u = 0$ with final data $e^{-f_{t_2}}$ on $[t_1,t_2]$.  From Lemma \ref{l:conjheat} part (1) and Proposition \ref{p:steadyharnack} we conclude on the time interval $[t_1,t_2]$,
\begin{align*}
    \frac{d}{dt} \eh(\gm, \gs^f) =&\ \frac{d}{dt} \int_M \gr(\gm, \gs^f) u \gs^2 = \int_M \left[ - \square^*_{\sigma} \left( \gr(\gm, \gs^f) u \right) \right] \gs^2 = \int_M \brs{\grc(\gm, \gs^f)}^2_{\gm} u\sigma^2.
\end{align*}
In particular this derivative is nonnegative, and we conclude
\begin{align*}
\gl(\gm_{t_1}) \leq \eh(\gm_{t_1},\gs_{t_1}^f) \leq \eh(\gm_{t_2},\gs_{t_2}^f) = \gl(\gm_{t_2}).
\end{align*}
Thus the monotonicity follows, and in the case of equality the computation above shows that $\grc(\gm, \gs^f)$ vanishes, and thus $\gm$ is a soliton.
\end{proof}

\subsection{Convergence of nonsingular solutions}

\begin{defn} \label{d:nonsingular} We say a solution $(\gm_t, \gs_t)$ to generalized Ricci flow is \emph{nonsingular} if it exists on $[0, \infty)$ and for any sequence of times $\{t_i\} \to \infty$ there exists a subsequence still denoted $\{t_i\}$, Courant automorphisms $\Phi_{t_i}$, and $\ga_{t_i} \in \mathbb R_{>0}$ such that the sequence $(\Phi_{t_i}^* \gm_{t_i}, \ga_{t_i} \Phi_{t_i}^* \gs_{t_i})$ converges to a limiting pair $(\gm_{\infty}, \gs_{\infty})$.
\end{defn}

\begin{rmk} Note that the notion of convergence in this definition is a natural extension of the notion of Cheeger--Gromov convergence of Riemannian metrics to the setting of generalized metrics on Courant algebroids.  Furthermore, it implies that all invariant objects constructed from $(\gm, \gs)$ such as Ricci and scalar curvature etc.\ are bounded.  The scaling factors $\ga$ for the half-densities are necessary since generalized Ricci flat structures can have positive scalar curvature, and so one only expects the generalized Ricci flow of such objects to be convergent up to scaling the half-density.  Lastly we note that in the Ricci flow literature the term `nonsingular' is usually weaker than this, only imposing a uniform bound on the Riemannian curvature tensor.
\end{rmk}

\begin{cor} \label{c:convcor2} (cf.\ Corollary \ref{c:convcor}) Every nonsingular solution to generalized Ricci flow with strictly positive initial data converges subsequentially to a soliton.
\begin{proof} We first note that by the assumption of the solution being nonsingular it follows that all Courant automorphism-invariant quantities are uniformly bounded for all times.  In particular $\gl$ is monotonically increasing and bounded above, and moreover all of its time derivatives are uniformly bounded.  It follows that $\lim_{t \to \infty} \frac{d}{dt} \gl = 0$.  Fix now any sequence $\{t_i \to \infty\}$, then by hypothesis choose the relevant subsequence, still denoted $\{t_i\}$, and Courant automorphisms $\Phi_{t_i}$ and $\ga_{t_i}$ so that the sequence $(\til{\gm}_i, \til{\gs}_i) := (\Phi_{t_i}^* \gm_{t_i}, \ga_{t_i} \Phi^*_{t_i} \gs_{t_i})$ converges to a limit $(\gm_{\infty}, \gs_{\infty})$.
Now let $\til{\gs}_{i}$ denote the unique function achieving the infimum in the definition of $\gl(\til{\gm}_i)$.  It follows from the remarks above that $\grc(\gm_{\infty}, \gs_{\infty}) = \lim_{i \to \infty} \grc(\til{\gm}_i, \til{\gs}_i) = 0$, as claimed.
\end{proof}
\end{cor}

\section{Remarks and further questions} \label{s:outlook}

\begin{rmk}[T-duality and its generalizations]\label{r:Tduality}
Physical interpretation of the generalized Ricci flow as the one-loop renormalisation group flow of the string nonlinear sigma model suggests a nontrivial compatibility between the flow and the so-called \emph{T-duality} phenomenon. The latter admits a natural description in terms of Courant algebroids \cite{GualtieriTdual}. Following this interpretation the compatibility was shown in the case of exact Courant algebroids in \cite{BarHek,StreetsTdual,GF19}. This was later generalized in \cite{GF19} to the case where the duality corresponds to a pair of (arbitrary) \emph{$T^n$-equivariant Courant algebroids} producing the same quotient.

A wide (non-abelian) generalization of T-duality is provided by the \emph{Poisson--Lie T-duality} \cite{Klimcik:1995ux}, which in its most general form \cite{Klimcik:1996np} can be conveniently described in terms of \emph{Courant algebroid pullbacks} and \emph{equivariant reductions} \cite{Severa,SeveraPLTDuality,SeveraValach1,SeveraValach2}. The full compatibility between Poisson--Lie T-duality and the generalized Ricci flow was established in \cite{SeveraValach2}. 
\end{rmk}

\begin{rmk}[Courant algebroids with point base]
    Related to the previous remark, let us briefly comment on the setup of Example \ref{ex:trivial}, corresponding to $M=\text{point}$, in which case $E=\mf g$ is a quadratic Lie algebra. We then have $\ms H^*\cong \mb R^+$, with all half-densities inducing the trivial divergence $\div=0$. For such a divergence we get the algebraic expressions
    \begin{align*}
    \grc\indices{^\alpha_\beta}&=\tfrac14(c^{\alpha\gamma\delta}c\indices{_\beta_\gamma_\delta}-\gm\indices{^\alpha_\gamma}\gm\indices{^\delta_\beta}c^{\gamma\epsilon\zeta}c_{\delta\epsilon\zeta}-\gm\indices{^\gamma_\delta}\gm\indices{^\epsilon_\zeta}c^{\alpha\delta\zeta}c_{\beta\gamma\epsilon}+\gm\indices{^\alpha_\gamma}\gm\indices{^\delta_\beta}\gm\indices{^\eta_\epsilon}\gm\indices{^\theta_\zeta}c^{\gamma\epsilon\zeta}c_{\delta\eta\theta})\\
    \gr&=\tfrac14\gm\indices{^\alpha_\delta}c_{\alpha\beta\gamma}c^{\delta\beta\gamma}-\tfrac1{12}\gm\indices{^\alpha_\delta}\gm\indices{^\beta_\epsilon}\gm\indices{^\gamma_\zeta}c_{\alpha\beta\gamma}c^{\delta\epsilon\zeta},
    \end{align*}
    where $c_{\alpha\beta\gamma}$ are the structure constants of $\mf g$. In the generalized Ricci flow the evolution of $\sigma$ decouples from that of $\gm$, which is simply the ODE
    \[\frac{d}{dt}\gm=-2\grc,\]
    with $\div=0$, on the now finite-dimensional space $\ms M$. Various results from the present text still apply, in particular the curvature evolution along the generalized Ricci flow 
    \[\frac{d}{dt}\gr=\vert\!\grc\!\vert_{\gm}^2,\]
    or the Bianchi identity, which simplifies to $D_\alpha \grc^{\alpha\beta}=0$ for any Levi-Civita connection $D$.
    The interest in such an algebraic setup comes from the fact that fixed points of this flow correspond under suitable conditions \cite{SeveraValach2} to fixed points of the generalized Ricci flow on particular exact Courant algebroids. The results of the present text might thus lead to insights into the nature of such solutions and perhaps to an efficient procedure for their search.
\end{rmk}

\begin{ques} A natural question is what the basic regularity requirements are to extend a solution to generalized Ricci flow past a certain time.  Here one expects a parabolic regularity theory centered around the curvature tensor and its derivatives as in the case of Ricci flow.  Here a further wrinkle is present due to the nonuniqueness of the Levi-Civita connection, and one likely needs to choose an initial Levi-Civita connection and use the connections constructed in Proposition \ref{prop:defconnection} to obtain smoothing estimates and clear regularity criteria.
\end{ques}

\begin{ques} A basic structural property of generalized Ricci flow on exact Courant algebroids (cf.\ \ref{f:exactGRF}) and on certain transitive Courant algebroids obtained by reduction (cf.\ \cite{GF19}) is that the metric is a \emph{supersolution} to Ricci flow.  This yields immediately many a priori estimates holding for such supersolutions \cite{McCannTopping,BamlerHeatKernel}.  A natural question is if this property holds in general.  Similarly, as shown in \cite{Streetsscalar}, in the exact case, the dilaton flow is a supersolution to the heat equation, again leading to an important a priori estimate.  Furthermore, upper bounds on the dilaton lead to Perelman entropy functionals and $\gk$-noncollapsing results.  Understanding better the precise structure of the dilaton flow in general should lead to extensions of these key properties.
\end{ques}

\begin{ques}
    Transitive Courant algebroids sit inside a larger class of \emph{regular} Courant algebroids, whose defining condition is that $\rho$ has constant rank. Although in the non-transitive case such algebroids do not admit generalized metrics as defined in the present text, one can try to relax the corresponding definition by allowing all generalized pseudometrics for which $\langle\cdot,\cdot\rangle|_{V_+}$ is positive definite and $\rho$ restricts to an isomorphism between $V_+$ and the subbundle $\rho(E)\subset TM$. Since $\rho(E)$ is integrable, the base manifold then decomposes into integral leaves, each of which carries a transitive Courant algebroid with a generalized metric in the usual sense. Generalized Ricci flow in this context thus corresponds to the study of families of generalized Ricci flows on transitive Courant algebroids. A natural question is to investigate basic properties of such flow, and perhaps even extend some results to the case of arbitrary (not necessarily regular) Courant algebroids.
\end{ques}

\begin{ques} An interesting question is to further understand the relationship of generalized Ricci flow to complex geometry.  In the exact CA case it was shown in \cite{PCFReg} that the generalized Ricci flow preserves the data of a pluriclosed Hermitian metric, after coupling to a flow of the complex structure.  As shown in \cite{GKRF} the flow even preserves the generalized K\"ahler condition \cite{GualtieriGKG}.  Furthermore, in \cite{JGFS} the pluriclosed flow was shown to be naturally interpreted as a flow of  metrics on a \emph{holomorphic} Courant algebroid, and the precise relationship of the two formulations remains slightly unclear.  Another interesting problem is to show in a conceptual way intrinsic to generalized geometry why these conditions are preserved in the first place.

One would like to extend this relationship to more general Courant algebroids.  In forthcoming work \cite{GFGMS} it is shown that the generalized Ricci flow of \cite{GF19} preserves a certain ansatz in complex geometry arising from the Hull--Strominger system, and admits a further interpretation as a flow of Hermitian metrics on a certain holomorphic Courant algebroid, extending the results of \cite{JGFS}, and leading to a number of new a priori estimates.  Recent work \cite{CortesBnGK} gives an extension of generalized K\"ahler geometry to the odd-type Courant algebroids of \cite{RubioThesis}.  It seems natural to expect that this geometric structure is preserved by generalized Ricci flow.
\end{ques}

\bibliography{Streets_Master}
\bibliographystyle{plain}

\end{document}